\documentclass[11pt,a4paper]{article}

\usepackage[english]{babel}
\usepackage[utf8]{inputenc}
\usepackage[T1]{fontenc}
\usepackage{lmodern}

\usepackage{graphicx}
\graphicspath{{assets/}}
\usepackage{graphbox}
\usepackage{amssymb}
\usepackage{amsmath}
\usepackage{amsthm}
\usepackage{mathtools}
\usepackage{stmaryrd}
\usepackage[colorlinks=true,linkcolor=blue,citecolor=blue,urlcolor=blue]{hyperref}
\usepackage{cleveref}

\usepackage{csquotes,xpatch}
\usepackage[style=numeric,giveninits=true,sorting=nyt,maxbibnames=99,backend=biber]{biblatex}
\DeclareFieldFormat{titlecase}{\MakeSentenceCase*{#1}}

\usepackage[a4paper,margin=1in]{geometry}

\usepackage{enumitem}
\usepackage{microtype}

\input{macros.sty}

\newtheorem{theorem}{Theorem}[section]
\newtheorem{lemma}[theorem]{Lemma}
\newtheorem{corollary}[theorem]{Corollary}
\newtheorem{proposition}[theorem]{Proposition}

\theoremstyle{definition}

\theoremstyle{remark}
\newtheorem{remark}[theorem]{Remark}

\title{Strong edge-colouring via local flag algebras}
\author{Eoin Davey \and Eoin Hurley \and R\'emi de Joannis de Verclos \and Ross J. Kang \and Jan Volec}
\date{\today}

\begin{document}

\maketitle

\begin{abstract}
The strong chromatic index $\chi'_s(G)$ is the smallest number of
colours needed to colour the edges of a graph $G$ so that any two
edges at distance at most $2$ receive different colours. Using the
\emph{local flag algebra} framework introduced in a companion paper, we prove
$\chi'_s(G) \leq 1.73\,\Delta(G)^2$ for every graph $G$ of maximum
degree $\Delta(G)$, $\chi'_s(G) \leq 1.6255\,\Delta(G)^2$ for every
bipartite $G$, and $\chi'_s(G) \leq 1.6633\,\Delta_A(G)\,\Delta_B(G)$
for every bipartite $G$ of side maximum degrees
$\Delta_A(G), \Delta_B(G)$ with rational
$\Delta_B(G)/\Delta_A(G) \in (0, 1]$, provided $\Delta(G)$, $\Delta_A(G)$, $\Delta_B(G)$ are sufficiently large. These three bounds make progress towards three established
conjectures: those of Erd\H{o}s--Ne\v{s}et\v{r}il (1985) for general
graphs, Faudree--Gy\'arf\'as--Schelp--Tuza (1989) for bipartite graphs,
and Brualdi--Quinn Massey (1993) in the asymmetric bipartite setting.

Additionally, for the random bipartite graph $G \sim G(n_A, n_B, p)$
at constant $p \in (0,1)$ and bounded aspect ratio
$\max(n_A, n_B) = O(\min(n_A, n_B))$, we prove the
Brualdi--Quinn Massey bound
$\chi'_s(G) \leq \Delta_A(G)\,\Delta_B(G)$ asymptotically almost
surely.
\end{abstract}


\section{Introduction}
\label{sec:introduction}

A \emph{strong edge-colouring} of a graph $G$ is a proper edge-colouring in which any two edges sharing a vertex or joined by a
third edge receive different colours. Equivalently, it is a proper
vertex-colouring of the square $L(G)^2$ of the line graph $L(G)$, where
$V(L(G)^2) = E(G)$ and two edges are adjacent in $L(G)^2$ when they
sit at distance at most $2$ in $L(G)$. The minimum number of colours
needed is the \emph{strong chromatic index} $\chi'_s(G)$.

Erd\H{o}s and Ne\v{s}et\v{r}il, in
1985~\cite{faudreeInducedMatchingsBipartite1989}, conjectured that
$\chi'_s(G) \leq \tfrac{5}{4}\Delta(G)^2$ for every graph $G$; this is sharp
on a blow-up of $C_5$. After the breakthrough by Molloy and
Reed~\cite{molloyBoundStrongChromatic1997}~($1.998\,\Delta^2$), further successive improvements over the trivial bound
$2\Delta(G)^2$ were due to
Bruhn and Joos~\cite{bruhnStrongerBoundStrong2018}~($1.93\,\Delta^2$),
Bonamy, Perrett, and
Postle~\cite{bonamyColouringGraphsSparse2018}~($1.835\,\Delta^2$),
and Hurley, de~Joannis~de~Verclos, and
Kang~\cite{hurleyImprovedProcedureColouring2022}~($1.772\,\Delta^2$),
each for $\Delta(G)$ sufficiently
large. For bipartite $G$, Faudree,
Gy\'arf\'as, Schelp, and
Tuza~\cite{faudreeInducedMatchingsBipartite1989} conjectured the
sharper bound $\chi'_s(G) \leq \Delta(G)^2$, and Brualdi and
Quinn Massey~\cite{brualdiIncidenceStrongEdge1993} conjectured an
asymmetric strengthening $\chi'_s(G) \leq \Delta_A(G)\,\Delta_B(G)$
for bipartite $G$ with side maximum degrees $\Delta_A(G), \Delta_B(G)$. Prior to
the present work, no broad improvement over the general bound was known
in either the symmetric or the asymmetric bipartite setting.

In a companion
paper~\cite{daveyPentagonLocalFlags2026}, we introduced \emph{local flag
algebras}, a novel variant of Razborov's flag
algebras~\cite{razborovFlagAlgebras2007} in which densities are
normalised by the maximum degree $\Delta(G)$ rather than the order
$|G|$. Applying that framework to the
strong-neighbourhood density of $L(G)^2$ yields our three main results.
These results constitute progress towards the Erd\H{o}s--Ne\v{s}et\v{r}il, Faudree--Gy\'arf\'as--Schelp--Tuza,
and Brualdi--Quinn Massey conjectures, respectively.

\begin{theorem}[general strong chromatic index bound]
\label{thm:sec-general}
For every graph $G$ with $\Delta(G)$ sufficiently large,
\[
\chi'_s(G) \leq 1.73\,\Delta(G)^2.
\]
\end{theorem}

\begin{theorem}[bipartite strong chromatic index bound]
\label{thm:sec-bipartite}
For every bipartite graph $G$ with $\Delta(G)$ sufficiently large,
\[
\chi'_s(G) \leq 1.6255\,\Delta(G)^2.
\]
\end{theorem}

\begin{theorem}[asymmetric strong chromatic index bound]
\label{thm:asym-sec}
For every rational $r \in (0, 1]$ and every bipartite graph $G$
with side maximum degrees $\Delta_A$ and $\Delta_B$ satisfying
$\Delta_B = r\,\Delta_A$, and $\Delta_A$ sufficiently large,
\[
\chi'_s(G) \leq 1.6633 \cdot \Delta_A \Delta_B.
\]
\end{theorem}
\noindent
The last theorem is stated for rational $r$; for irrational $r$ no
bipartite graph realises $\Delta_B = r\,\Delta_A$ exactly, but the
same bound holds asymptotically along any sequence with
$\Delta_B/\Delta_A \to r$ and $\Delta_A \to \infty$.

Additionally, we confirm the Brualdi--Quinn Massey conjecture itself,
asymptotically almost surely (a.a.s.), for the random bipartite graph.

\begin{theorem}[a.a.s.\ Brualdi--Quinn Massey]
\label{thm:bq-random}
Let $p \in (0, 1)$ and $C_{\mathrm{ar}} \geq 1$. Then for
$G \sim G(n_A, n_B, p)$ of bounded aspect ratio
$\max(n_A, n_B) \leq C_{\mathrm{ar}} \cdot \min(n_A, n_B)$,
\[
\chi'_s(G) \leq \Delta_A(G) \cdot \Delta_B(G)
\qquad\text{asymptotically almost surely as }
\min(n_A, n_B) \to \infty.
\]
\end{theorem}

We outline the proofs. 
Theorems~\ref{thm:sec-general}, \ref{thm:sec-bipartite},
and~\ref{thm:asym-sec} are the local-flag-algebra applications; their
proofs follow the architecture established for the local pentagon problem
in~\cite{daveyPentagonLocalFlags2026} and adapted to the
strong edge-colouring problem through the reduction to
strong-neighbourhood density.
The proof of Theorem~\ref{thm:asym-sec} reduces every rational ratio
to a single regular class --- by regularising both sides and then
balancing the two side degrees by a blow-up --- on which one
four-colour semidefinite-programming (SDP) certificate applies.
Theorem~\ref{thm:bq-random} confirms the
Brualdi--Quinn Massey conjecture~\cite{brualdiIncidenceStrongEdge1993}
for random bipartite graphs of bounded aspect ratio, through Kahn's
Pippenger--Spencer covering argument~\cite{kahnAsymptoticEdgeColoring1996}
fed by the Kim--Vu polynomial concentration
inequality~\cite{kimVuConcentration2000}.

\paragraph{Note on AI and Lean.}

A desired standard in the flag algebra community is, for robustness and increased confidence, to generate the needed SDP in two different software implementations.
Here our approach is new: we formally verify in Lean~4 the proof that our flag-algebra implementation produces, rather than relying on a separate codebase for corroboration.
We also wish to disclose that we obtained Theorem~\ref{thm:bq-random}, intended as an auxiliary supporting result toward the Brualdi--Quinn Massey conjecture, by deploying a commercially available agentic AI system to construct its proof directly in Lean~4 under our guidance.
See the closing notes \textit{\nameref{sec:lean-and-empirical}} and \textit{\nameref{sec:AI-usage}} for additional details.

\paragraph{Organisation.}

Section~\ref{sec:framework} fixes the notation and the framework
results we import from the companion
paper~\cite{daveyPentagonLocalFlags2026}; the companion paper proves
the framework itself, and we re-use it as a black box, restating only
what we need. Section~\ref{sec:problem} sets up the strong edge-colouring
problem, the reduction to strong-neighbourhood density, and the reduction to the regular case
(Lemma~\ref{lem:wlog-regular}). Section~\ref{sec:general} proves
Theorem~\ref{thm:sec-general} and compares it with the previous best
bound. Section~\ref{sec:bipartite} proves
Theorem~\ref{thm:sec-bipartite}. Section~\ref{sec:asym-sec} sets
up the asymmetric four-colour local-flag SDP, gives the reduction to
the biregular case and the degree-balancing blow-up, and obtains
Theorem~\ref{thm:asym-sec}.
Section~\ref{sec:certificates} describes the three SDP certificates
(general, bipartite, and asymmetric) as mathematical objects and
verifies the bridging lemmas left open by the main-text proofs.
Section~\ref{sec:random} proves Theorem~\ref{thm:bq-random} by an
auxiliary-hypergraph Pippenger--Spencer argument at constant matching
size.

\paragraph{Note added.}
During the preparation of this paper --- which builds in part on Eoin Davey's 2024 MSc thesis~\cite{daveyLocalFlags2024} (see~\url{https://scripties.uba.uva.nl/search?id=record_54581}) --- we learned of the independent work of Hao, Yang, and Yu~\cite{haoStrongChromaticIndexBipartite2026}, who announced a weaker form of Theorem~\ref{thm:asym-sec}.

\section{Local flag algebras: recap}
\label{sec:framework}

We import the framework
of~\cite{daveyPentagonLocalFlags2026} verbatim; the only purpose of
this section is to fix notation. The reader familiar with the
companion paper may skip to Section~\ref{sec:problem}.

Throughout, $\Gcl$ denotes a graph class, possibly with vertex and
edge colours, and $\Delta\colon \Gcl \to \mathbb{N}_0$ denotes the
maximum degree function. The hereditary closure
$\HeredG \supseteq \Gcl$ is closed under taking induced subgraphs.

\subsection{Types and flags}
A \emph{type} $\sigma$ is a (vertex- and edge-coloured) graph on the
ordered vertex set $[|\sigma|]$. A $\sigma$-flag $(F, \theta)$
consists of $F \in \HeredG$ together with an injection
$\theta\colon [|\sigma|] \to V(F)$ that realises $\sigma$ as the
induced subgraph $F[\im\theta]$ with the prescribed colours. We
write $\Gcl^\sigma$ for the isomorphism classes of $\sigma$-flags
drawn from $\Gcl$, and $\Gcl_n^\sigma$ for those of size $n$. For
$(F,\theta) \in \HeredG^\sigma$ and $(G,\eta) \in \Gcl^\sigma$, let
$c(F;G)$ denote the number of subsets $\im\eta \subseteq U \subseteq V(G)$
with $G[U] \cong F$ as $\sigma$-flags. The \emph{local density} is
\[
\rho(F; G) := \frac{c(F; G)}{\binom{\Delta(G)}{|F| - |\sigma|}}.
\]

\subsection{Local flags}
A $\sigma$-flag $(F, \theta)$ is a \emph{local} $\sigma$-flag if
both $\rho(F; \cdot)$ is bounded on $\Gcl^\sigma$ and every
label-extension $F^v$ (obtained by labelling one unlabelled vertex
$v$) is also a local flag. We write $\Gloc^\sigma$ for the local
flags and $\Glocsub{n}^\sigma$ for those of size $n$. A type
$\sigma$ is a \emph{local type} when the unlabelled graph
$\downflag{\sigma}$ is itself a local $\emptyset$-flag.

\subsection{The local flag algebra}
The vector space $\R\Gloc^\sigma$, modulo the chain-rule identities,
carries a commutative associative unital product (the \emph{local
flag product}); we denote the resulting algebra by $\Lcl^\sigma$.
The product satisfies
\[
\rho(f; G)\,\rho(g; G) = \rho(f \cdot g; G) + O(1/\Delta(G))
\]
for $f, g \in \Lcl^\sigma$ and $G \in \Gcl^\sigma$. A \emph{limit
functional} $\phi \in \Phi^\sigma$ is an $\R$-algebra homomorphism
$\Lcl^\sigma \to \R$ arising as $\phi(F) = \lim_k \rho(F; G_k)$ for
a $\Delta$-increasing sequence $(G_k) \subseteq \Gcl^\sigma$. The
\emph{semantic cone} is
\[
\SemCone^\sigma := \{ f \in \Lcl^\sigma : \phi(f) \geq 0
\text{ for all } \phi \in \Phi^\sigma \}.
\]

\subsection{Averaging}
For a local type $\sigma$, the averaging operator
$\llbracket \cdot \rrbracket\colon \Lcl^\sigma \to \Lcl^\emptyset$
sends $F$ to $\llbracket F \rrbracket := q_\sigma(F)\,\downflag{F}$
with $q_\sigma(F) = |\mathrm{Stab}(\theta)| / |F|!$. It preserves
positivity: $\llbracket f^2 \rrbracket \in \SemCone^\emptyset$ for
every $f \in \Lcl^\sigma$.

\subsection{Extension vectors}
When $\Gcl$ consists of regular graphs, the \emph{extension vector}
$\ext_i^\sigma$ is the sum over $\sigma$-flags of size $|\sigma|+1$
whose unique unlabelled vertex is adjacent to the vertex labelled
$i$. Then $\phi(\ext_i^\sigma) = 1$ for every $\phi \in \Phi^\sigma$,
and consequently
$\llbracket \ext_i^\sigma - \ext_j^\sigma \rrbracket \in
\SemCone^\emptyset$ and
$\phi(\llbracket f \cdot \ext_i^\sigma \rrbracket) = \phi(\llbracket f
\rrbracket)$ for every $f \in \Lcl^\sigma$ and every local type
$\sigma$.

\subsection{The semidefinite method}
To certify $\phi(f) \leq \lambda$ for $f \in \Lcl^\emptyset$ and
every $\phi \in \Phi^\emptyset$ it suffices to exhibit a
decomposition
\begin{equation}
\label{eq:sdp-decomp}
\lambda \emptyset - f
= \sum_i \alpha_i\,g_i
+ \sum_j \bigl\llbracket h_j^2 \bigr\rrbracket
\end{equation}
in $\Lcl^\emptyset$, with $\alpha_i \geq 0$, $g_i \in
\SemCone^\emptyset$ from the extension-difference family
$\llbracket \ext_p^\tau - \ext_q^\tau \rrbracket$ and any class-specific
linear constraints, and $h_j \in \Lcl^{\sigma_j}$ for local types
$\sigma_j$. The decomposition is the primal feasibility witness of a
semidefinite program; once an SDP solver has computed a primal-dual
pair, verifying~\eqref{eq:sdp-decomp} amounts to a finite
arithmetic check on integer or rational data.

We follow the certificate-verification convention of the companion
paper throughout
Sections~\ref{sec:cert-general}--\ref{sec:cert-asym}: exhibit the
matrices, give an integer $LDL^\top$ identity for each positive
semidefinite (PSD) block, write the
constraint identity as a polynomial equation in the certificate coefficients,
and absorb the rationalisation residual into an explicit slack budget.

\section{Strong edge-colouring and a regularity reduction}
\label{sec:problem}

\subsection{A reduction to strong-neighbourhood density}

For $e \in E(G)$, the \emph{strong neighbourhood} $N_s(e)$ is the
set of edges of $G$ at distance at most $2$ in $L(G)$. A strong
edge-colouring is a proper vertex-colouring of $L(G)^2$, so
$\chi'_s(G) = \chi(L(G)^2)$. Every published improvement on the
trivial bound comes from a sparsity bound on
$|E(L(G)^2[N_s(e)])|$ fed into a probabilistic colouring lemma;
the present paper improves the sparsity input via local flag
algebras.

The relevant probabilistic black box is a sparse colouring lemma,
stated for the following sparsity notion: a graph $H$ is
\emph{$\varsigma$-sparse} if every vertex neighbourhood spans few
edges, $|E(H[N_H(v)])| \leq (1 - \varsigma)\binom{\Delta(H)}{2}$ for
every $v \in V(H)$.

\begin{theorem}[sparse colouring lemma~{\cite{hurleyImprovedProcedureColouring2022}}]
\label{thm:hjk-colouring}
For every $\varsigma \in (0, 1]$ and $\iota > 0$ there is a
$\Delta_0 = \Delta_0(\varsigma, \iota)$ such that every
$\varsigma$-sparse graph $H$ with $\Delta(H) \geq \Delta_0$ satisfies
\[
\chi(H) \leq (1 - \varepsilon(\varsigma) + \iota)\,\Delta(H),
\qquad
\varepsilon(\varsigma) := \varsigma/2 - \varsigma^{3/2}/6.
\]
\end{theorem}

We use the lemma in a mildly more flexible form, measuring sparsity
against a stated degree scale rather than the exact maximum degree.

\begin{corollary}[degree-scale form]
\label{cor:hjk-scale}
For every $\varsigma \in (0, 1]$ and $\iota > 0$ there is a
$\Delta_0$ such that the following holds for every $D \geq \Delta_0$.
If $H$ is a graph with $\Delta(H) \leq D$ and
$|E(H[N_H(v)])| \leq (1 - \varsigma)\binom{D}{2}$ for every
$v \in V(H)$, then $\chi(H) \leq (1 - \varepsilon(\varsigma) + \iota)\,D$.
\end{corollary}

\begin{proof}
The disjoint union $H' := H \cup K_{D,D}$ has $\Delta(H') = D$, its
new vertices have independent neighbourhoods, and its old vertices
keep theirs, so $H'$ is $\varsigma$-sparse.
Theorem~\ref{thm:hjk-colouring} gives
$\chi(H) \leq \chi(H') \leq (1 - \varepsilon(\varsigma) + \iota)\,D$.
\end{proof}

We cannot apply Theorem~\ref{thm:hjk-colouring} to $H = L(G)^2$
directly. Bruhn and Joos~\cite{bruhnStrongerBoundStrong2018} showed
that a strong neighbourhood can span $\tfrac{3}{2}\Delta(G)^4 +
5\Delta(G)^3$ edges of $L(G)^2$ and that this is asymptotically best
possible, which caps the sparsity of $L(G)^2$ itself at
$\varsigma = 1/4$ --- too weak for the bounds we pursue. (The
companion paper reproduces exactly this $\tfrac{3}{2}$ from a size-$4$
local flag program.) Following Bonamy, Perrett and
Postle~\cite{bonamyColouringGraphsSparse2018}, as adapted by Hurley,
de~Joannis~de~Verclos, and Kang~\cite{hurleyImprovedProcedureColouring2022}, we circumvent the
obstacle by restricting to a subgraph of high minimum degree. Fix
$\eta$ and let $F \subseteq V(L(G)^2)$ be a maximal set on which the
induced subgraph has minimum degree at least $(2 - \eta)\Delta(G)^2$;
since the union of two such sets is again one, $F$ contains every
such set. The local flag SDP --- invoked after the regularity
reduction of Lemma~\ref{lem:wlog-regular} --- bounds the edge density
of the strong neighbourhoods \emph{within $F$}, and
Corollary~\ref{cor:hjk-scale} colours $L(G)^2[F]$ at the degree scale
$2\Delta(G)^2$. The colouring then extends to all of $L(G)^2$ by the
criticality argument from the proof
of~\cite[Theorem~1.6]{hurleyImprovedProcedureColouring2022}: if
$\chi(L(G)^2) > t$ for
$t := \max\bigl(\chi(L(G)^2[F]),\, \lceil(2 - \eta)\Delta(G)^2\rceil + 1\bigr)$,
then a minimal vertex set $F'$ with $\chi(L(G)^2[F']) > t$ would be
vertex-critical, hence of induced minimum degree at least
$t - 1 \geq (2 - \eta)\Delta(G)^2$, forcing $F' \subseteq F$ and the
contradiction $\chi(L(G)^2[F']) \leq \chi(L(G)^2[F]) \leq t$. In short,
\begin{equation}
\label{eq:criticality-lift}
\chi(L(G)^2)
\leq \max\bigl(\chi(L(G)^2[F]),\,
\lceil(2 - \eta)\Delta(G)^2\rceil + 1\bigr).
\end{equation}
We phrase the density input as a problem on a class of
$(2,2)$-coloured graphs in Section~\ref{sec:general} (general case)
and $(4,2)$-coloured bipartite graphs in
Section~\ref{sec:bipartite}, recording membership of $F$ by an edge
colour.

\subsection{A reduction to regular graphs}

The local flag SDP operates on a class of regular graphs. We give the
following standard reduction, which shows that the regularity
hypothesis costs nothing.

\begin{lemma}[regular suffices]
\label{lem:wlog-regular}
Let $c > 0$. If there is a $\Delta_0$ such that every
$\Delta$-regular graph $G'$ with $\Delta(G') \geq \Delta_0$ satisfies
$\chi'_s(G') \leq c\,\Delta(G')^2$, then there is a $\Delta_1$ such
that every graph $G$ with $\Delta(G) \geq \Delta_1$ satisfies
$\chi'_s(G) \leq c\,\Delta(G)^2$. The same holds with the additional
hypothesis ``bipartite'' on both sides.
\end{lemma}

\begin{proof}
Let $G$ be a graph of maximum degree $\Delta = \Delta(G)$. We
construct a $\Delta$-regular host $\widehat{G}$ of which $G$ is an
induced subgraph and such that $\chi'_s(G) \leq \chi'_s(\widehat{G})$;
the hypothesised bound on $\widehat{G}$ then pulls back to $G$.

Form $G_0 := G$ and $G_{i+1}$ from $G_i$ by taking a disjoint copy
$G_i'$ of $G_i$ on a fresh vertex set and adding an edge between each
\emph{deficient} vertex $v$ of $G_i$ (one with $\deg_{G_i}(v) < \Delta$)
and its copy $v' \in V(G_i')$. The construction preserves $\Delta$:
the added edges incident to $v$ raise $\deg(v)$ by at most one, and
$v$ was deficient. After at most $\Delta - \delta(G)$ iterations
every vertex is non-deficient, so $G_k$ is $\Delta$-regular for some
$k \leq \Delta$. The inclusion $G \subseteq G_k$ stays induced
because edges between $V(G_i)$ and $V(G_i')$ exist only between
corresponding vertices of the two copies. Hence
$\chi'_s(G) \leq \chi'_s(G_k)$. Setting $\widehat{G} := G_k$
finishes the proof.

The bipartite version uses the same construction; an added edge
$vv'$ runs between the two copies of $G$, which carry compatible
bipartitions (the second copy's bipartition swaps the
first's), so the added edge runs between opposite parts of the
combined bipartition. Hence $G_k$ is bipartite whenever $G_0$ is.
\end{proof}

Throughout
Sections~\ref{sec:general}--\ref{sec:asym-sec} we assume $G$ regular;
Lemma~\ref{lem:wlog-regular} (extended to the asymmetric setting in
Section~\ref{sec:asym-sec}) removes this assumption from the headline
statements.

\section{A general bound}
\label{sec:general}

We prove Theorem~\ref{thm:sec-general}. The strategy mirrors
Section~\ref{sec:problem}: reduce the strong-neighbourhood density to
a local-flag-algebra problem on a $(2,2)$-coloured class, define an
objective vector, and exhibit an SDP decomposition of
$\lambda \emptyset - O$ in $\Lcl^\emptyset_5$ at
$\lambda = 10.644$, which, fed into the sparse colouring lemma,
yields the headline bound.

\subsection{The strong-neighbourhood density input and a coloured reduction}

After the regularity reduction (Lemma~\ref{lem:wlog-regular}) we
assume $G$ is $\Delta$-regular. Following Section~\ref{sec:problem},
fix $\eta \in [0, 0.3]$ and a maximal $F \subseteq V(L(G)^2)$ on which
$L(G)^2[F]$ has minimum degree at least $(2 - \eta)\Delta(G)^2$; we
seek an upper bound on $|E(L(G)^2[N_{L(G)^2[F]}(f)])|$, uniform in
$f \in F$.

Given $(G, F, f)$ with $f = \{u, v\} \in F$, construct a
$(2,2)$-coloured graph $G'$ on $V(G') = V(G)$ and $E(G') = E(G)$ by
colouring a vertex \emph{black} if it lies in $N(u) \cup N(v)$ and
\emph{red} otherwise, and colouring an edge \emph{black} if it lies
in $F$ and \emph{red} otherwise. Then $u, v$ are black, $f$ is a
black edge with black endpoints, and $G'$ has at most $2\Delta$ black
vertices; every strong $F$-neighbour of $f$ is a black edge with an
endpoint in $N(u) \cup N(v)$, hence with a black vertex. The incident pairs are
negligible at the leading order: pairs $\{e, e'\}$ that are
\emph{incident} in $G$ (share a vertex) contribute $O(\Delta^3)$ in
total --- choose the shared vertex ($\leq 2\Delta$ choices, since it
must be incident to $f$ or to a neighbour of $f$) and then two
further endpoints ($\leq \Delta^2$ choices) --- so relative to the
leading $\Delta^4$ scale of $|E(L(G)^2[N_s(f)])|$ we may restrict to
non-incident strong-neighbour pairs at no cost.

Let $E_O(G')$ be the set of unordered pairs $\{e, e'\} \subseteq
E(G')$ of non-incident \emph{black} edges sharing a common incident
edge in $G'$, with at least one black vertex on each of $e$ and
$e'$. The two observations give
\begin{equation}
\label{eq:EO-bridge}
|E(L(G)^2[N_{L(G)^2[F]}(f)])|
= |E_O(G')| + o(\Delta(G)^4).
\end{equation}
The minimum degree property of $F$ also transfers: each black edge
$e$ of $G'$ has at least $(2 - \eta)\Delta^2$ strong neighbours that
are black edges, and discarding the $O(\Delta)$ of them incident to
$e$ leaves at least $(2 - \eta)\Delta^2 - O(\Delta)$ non-incident
ones. We work with the class $\Gcl$ of $\Delta$-regular
$(2,2)$-coloured graphs with at most $2\Delta$ black vertices in
which every black edge has at least $(2 - \eta)\Delta^2 -
o(\Delta^2)$ non-incident black strong neighbours (the asymptotic
form of the two properties above). Locality of
$\sigma$-flags in
$\Gcl$ obeys the same rule as for the local pentagon class: a
$\sigma$-flag is local if and only if every connected component
contains a labelled or a black vertex.

\subsection{An objective vector}

The local-flag-algebra objective decomposes by the type of the
two-edge configuration in $E_O$. Let $\Bcl(G') \subseteq E(G')$ be
the set of black edges with at least one black endpoint, and
partition $\Bcl(G') = \Bcl_1(G') \cup \Bcl_2(G')$ according to
whether the edge has one or two black endpoints. With
$\sigma_1 := \bredge$ (one black, one red, joined by a black edge)
and $\sigma_2 := \edge$ (two black, joined by a black edge) --- both
local types --- the density of pairs $\{e, e'\}$ with $e$ of type
$\sigma_i$ contributing to $E_O$ takes the form
$\rho(\llbracket D(\sigma_i)\rrbracket; G')$, where $D(\sigma_i)$
is the sum over size-$4$ local $\sigma_i$-flags whose two unlabelled
vertices form a black edge connected to the labelled positions.
The calculation is the combinatorial bridge identity of
Section~\ref{sec:cert-general} (Lemma~\ref{lem:cbi-gen}), the
two-colour specialisation of the four-colour
accounting that Section~\ref{sec:bipartite} uses. The conclusion is
that
\[
\rho\bigl(2\,\llbracket D(\bredge) \rrbracket
+ \llbracket D(\edge) \rrbracket;\, G'\bigr)
= \frac{2\,|E_O(G')|}{\binom{\Delta(G')}{2}^2} + o(1)
\]
as $\Delta(G') \to \infty$. To express the objective in a fixed
size, we lift each summand to size $n = 5$ via the
multiplication-by-extension identity
$\phi(\llbracket f \cdot \ext_i^\sigma \rrbracket) =
\phi(\llbracket f \rrbracket)$ of Section~\ref{sec:framework}, and
set
\[
O := 2\,\bigl\llbracket D(\bredge)\cdot
  (\ext_1^{\bredge})^{n - 4} \bigr\rrbracket
+ \bigl\llbracket D(\edge)\cdot
  (\ext_1^{\edge})^{n - 4} \bigr\rrbracket
\in \Lcl^\emptyset_5.
\]

\subsection{An SDP bound}

We exhibit an SDP decomposition of $\lambda\emptyset - O$ in
$\Lcl^\emptyset_5$ at $\lambda = 10.644$, packaged as the
following lemma; we prove it in Section~\ref{sec:cert-general}.

\begin{lemma}[general strong edge-colouring density bound]
\label{lem:sec-general-density}
\[
\tfrac{10644}{1000}\,\emptyset - O
\in \SemCone^\emptyset
\quad\text{in }\Lcl^\emptyset_5,
\]
where $\Gcl$ is the class above. In particular,
$\phi(O) \leq 10.644$ for every $\phi \in \Phi^\emptyset$.
\end{lemma}

Section~\ref{sec:cert-general} exhibits and verifies the explicit SDP
certificate. Its one non-cone-standard ingredient is a minimum degree
constraint $\phi(D'(\sigma_1)) \geq 2(2 - \eta)$ at $\eta := 0.2703$,
imposed on the objective's black-red type $\sigma_1$, where $D'(\sigma)$
is the \emph{strong-degree flag} --- the number of black
$L(G)^2$-edges adjacent to a type-$\sigma$ black edge. This is the
asymptotic form of the property defining $F$: black edges lie in $F$,
and $L(G)^2[F]$ has minimum degree at least $(2 - \eta)\Delta(G)^2$
by construction (Section~\ref{sec:problem}).

\begin{proof}[Proof of Theorem~\ref{thm:sec-general}]
By Lemma~\ref{lem:wlog-regular} it suffices to bound $\chi'_s(G)$ on
$\Delta$-regular graphs $G$ with $\Delta$ sufficiently large. Fix
$\eta := 0.2703$ and a maximal $F \subseteq V(L(G)^2)$ on which
$L(G)^2[F]$ has minimum degree at least $(2 - \eta)\Delta(G)^2$. For
an edge $f \in F$ let $G'$ be the $(2,2)$-coloured reduction at
$(G, F, f)$; by construction $G' \in \Gcl$.

Let $(G'_k)$ be any $\Delta$-increasing sequence in $\Gcl$ converging
to a limit functional $\phi \in \Phi^\emptyset$ whose objective density
agrees with $\rho(O; G')$ up to $o(1)$. By
Lemma~\ref{lem:sec-general-density}, $\phi(O) \leq 10.644$. Convert
back to an edge count by~\eqref{eq:EO-bridge}: since $\rho(O; G') =
2|E_O(G')|/\binom{\Delta}{2}^2 + o(1)$,
\[
\frac{|E_O(G')|}{\Delta(G)^4}
\leq \frac{10.644}{8} + o(1)
= 1.3305 + o(1),
\]
and~\eqref{eq:EO-bridge} gives
$|E(L(G)^2[N_{L(G)^2[F]}(f)])| \leq 1.3305\,\Delta(G)^4 +
o(\Delta(G)^4)$, uniformly in $f \in F$. Hence $H := L(G)^2[F]$
satisfies the hypotheses of Corollary~\ref{cor:hjk-scale} at the
degree scale $D := 2\Delta(G)^2$: every strong neighbourhood has at
most $2\Delta(G)^2$ edges, so $\Delta(H) \leq D$, and with
$\varsigma := 1 - 10.644/16$ we have $(1 - \varsigma)\binom{D}{2} =
1.3305\,\Delta(G)^4 - O(\Delta(G)^2)$, so the density bound above is
$(1 - \varsigma')$-sparsity at scale $D$ for any fixed
$\varsigma' < \varsigma$ once $\Delta$ is large; we absorb the
difference into $\iota$ below and keep writing $\varsigma$.

Corollary~\ref{cor:hjk-scale} then colours $H$: a short calculation
gives $\varepsilon(\varsigma) = \varsigma/2 - \varsigma^{3/2}/6
\approx 0.135095$, so for any fixed $\iota > 0$,
\[
\chi(L(G)^2[F])
\leq (1 - 0.135095 + \iota)\cdot 2\Delta(G)^2
< (1.72982 + 2\iota)\Delta(G)^2.
\]
Finally the criticality lift~\eqref{eq:criticality-lift} extends the
bound to all of $L(G)^2$: since $2 - \eta = 1.7297$,
\[
\chi(L(G)^2)
\leq \max\bigl((1.72982 + 2\iota)\Delta(G)^2,\,
\lceil 1.7297\,\Delta(G)^2\rceil + 1\bigr),
\]
and for $\iota$ chosen so that $2\iota < 0.00018$ both branches are
strictly below $1.73\Delta(G)^2$ for $\Delta$ large. Hence
$\chi'_s(G) = \chi(L(G)^2) \leq 1.73\,\Delta(G)^2$.
\end{proof}

\subsection{Comparison with the previous best bound}
\label{sec:hjk-compare}

The general-case bound of Theorem~\ref{thm:sec-general} improves the
previous best general bound $\chi'_s(G) \leq 1.772\,\Delta(G)^2$ (for sufficiently large $\Delta(G)$)
of \textcite{hurleyImprovedProcedureColouring2022} by $0.042 \Delta(G)^2$.
Both proofs feed the same sparse colouring lemma
(Theorem~\ref{thm:hjk-colouring}); the difference lies entirely in
the bound on the sparsity input
$|E(L(G)^2[N_s(f)])|/\Delta(G)^4$. The earlier procedure bounds this
input by a structural refinement of the
Bonamy--Perrett--Postle~\cite{bonamyColouringGraphsSparse2018}
local-density estimate, evaluated at a worst-case configuration
within their high-degeneracy reduction. The present paper bounds the
same input by a local-flag-algebra SDP, which optimises directly
over the set $\Phi^\emptyset$ of limit functionals on the coloured
class. Since $\Phi^\emptyset$ embeds into the combinatorial relaxation
underlying the earlier procedure, the SDP returns a strictly smaller
bound; the gap measures
the slack of the combinatorial relaxation.

Three components of the SDP's tighter cone account for the
improvement.

\begin{enumerate}
\item The cone imposes the square-positivity (PSD block) constraints
$\phi(\llbracket f^2 \rrbracket) \geq 0$ for $f$ ranging over
$\Lcl_m^\sigma$ at every local type $\sigma$ and matching size $m$.
The earlier procedure uses no PSD constraint; this is the principal source of the
improvement, accounting for $\approx 0.03$--$0.04\,\Delta(G)^2$.
\item It also enforces extension-vector regularity
$\llbracket \ext_i^\sigma - \ext_j^\sigma \rrbracket \in
\SemCone^\emptyset$, which encodes the asymptotic regularity of the
input graph across all local types $\sigma$ in the SDP basis. The
earlier procedure uses regularity only globally; the local-type lift adds
$\approx 0.005$--$0.01\,\Delta(G)^2$.
\item It further imposes the black-vertex count constraint
$\phi(\ext_{B,1}^\sigma) \leq 2$ at every local type, encoding
$|B(G)| \leq 2\Delta(G)$ in the reduction. The earlier procedure uses an analogue only
at the level of $|F|$ itself; the local-type lift adds
$\approx 0.002$--$0.005\,\Delta(G)^2$.
\end{enumerate}

The three components interact, so the sum is suggestive rather than
exact. As an empirical consistency check, solving the SDP with only the
square-positivity and the minimum degree constraints (no
extension-difference, no black-vertex count) returns a bound
$\approx 1.752\,\Delta(G)^2$, which sits between the previous constant and
the present one. We do not pursue further decomposition; the headline
result is the strict $1.73\,\Delta(G)^2$ obtained from the full
constraint set.

\section{A bipartite bound}
\label{sec:bipartite}

We prove Theorem~\ref{thm:sec-bipartite}. The argument mirrors
Section~\ref{sec:general} but uses a four-colour palette to track
the bipartition combinatorially, expanding the SDP basis and the
objective vector.

\subsection{A bipartite coloured reduction}

Given a bipartite graph $G$ with parts $A, B$, both $\Delta$-regular,
together with $\eta \in [0, 0.4]$, a maximal $F \subseteq V(L(G)^2)$
as in Section~\ref{sec:general}, and an edge $f = \{u, v\} \in F$
with $u \in A$, $v \in B$, construct a $(4,2)$-coloured graph $G''$
by
\begin{itemize}
\item $N(u) \subseteq B$ red; $N(v) \subseteq A$ black;
\item remaining $A$-vertices blue; remaining $B$-vertices green;
\item edges of $F$ black; other edges red.
\end{itemize}
Unlike the general reduction of Section~\ref{sec:general} (where red
marks a non-$f$-side vertex), here both black ($N(v)$) and red ($N(u)$)
are $f$-side neighbourhoods; the colour names are local to each
reduction.
The encoding now stores the bipartition as
$\{\text{black, blue}\} \cup \{\text{red, green}\}$. The class
$\Gcl_{\mathrm{bip}}$ of $(4,2)$-coloured graphs respecting this
encoding, with regular bipartition components of common degree
$\Delta(G'')$, exactly $\Delta(G'')$ black vertices, exactly
$\Delta(G'')$ red vertices, is the bipartite analogue of
Section~\ref{sec:general}'s $\Gcl$. Locality in $\Gcl_{\mathrm{bip}}$ obeys the rule
``every connected component contains a labelled, black, or red
vertex'' (the proof matches Section~\ref{sec:general}; the
additional ``red'' anchor reflects that red vertices are
neighbour-of-$f$ and hence individually constrained).

The same accounting as in Section~\ref{sec:general} gives the
bipartite analogue
\[
|E(L(G)^2[N_{L(G)^2[F]}(f)])|
= |E_O(G'')| + o(\Delta(G)^4),
\]
where $E_O(G'')$ collects unordered pairs of non-incident black
edges sharing a common incident edge with at least one black or red
vertex on each, and the minimum degree property of $F$ transfers to
the class as in Section~\ref{sec:general}: every black edge has at
least $(2 - \eta)\Delta^2 - o(\Delta^2)$ non-incident black strong
neighbours.

\subsection{A bipartite objective}

The bipartite class admits three local edge types contributing to
$E_O$:
\[
\sigma_1' := \bredge,
\quad
\sigma_2' := \blredge,
\quad
\sigma_3' := \bgedge,
\]
where $\blredge$ has one blue and one red labelled vertex joined by a
black edge (a blue-red edge whose red endpoint lies on the $f$-side
neighbourhood) and $\bgedge$ has one black and one green vertex. With $D(\sigma)$ the
sum of size-$4$ local $\sigma$-flags whose two unlabelled vertices
form a black edge connected to the labelled positions and at least
one of them is on the $f$-side (black or red), the
bipartite objective is
\[
O := \bigl\llbracket D(\bredge)\cdot(\ext_1^{\bredge})^{n - 4}
\bigr\rrbracket
+ \bigl\llbracket D(\blredge)\cdot(\ext_1^{\blredge})^{n - 4}
\bigr\rrbracket
+ \bigl\llbracket D(\bgedge)\cdot(\ext_1^{\bgedge})^{n - 4}
\bigr\rrbracket
\in \Lcl^\emptyset_5.
\]

The bipartite analogue of the combinatorial bridge identity is
\[
\rho(O; G'')
= \frac{|E_O(G'')|}{\binom{\Delta(G'')}{2}^2} + o(1).
\]
Section~\ref{sec:cert-bipartite} proves this identity, the size-$5$
basis of cardinality $\ell_{\mathrm{bip}} = 3{,}808$ in
$\Gcl_{\mathrm{bip}}$, and the locality of every basis element, all
by the same certificate-verification template as in the general case.

\subsection{A bipartite SDP bound}

\begin{lemma}[bipartite strong edge-colouring density bound]
\label{lem:sec-bip-density}
At $\eta = 0.3746$,
\[
\tfrac{4093}{1000}\emptyset - O
\in \SemCone^\emptyset
\quad\text{in }\Lcl^\emptyset_5
\text{ on the bipartite class } \Gcl_{\mathrm{bip}}.
\]
In particular, $\phi(O) \leq 4.093$ for every limit functional on
$\Gcl_{\mathrm{bip}}$.
\end{lemma}

Again the proof reduces to certificate verification, in
Section~\ref{sec:cert-bipartite}. The bipartite-specific inputs are the
black-and-red vertex-count relations $\phi(\ext_{B,1}^\sigma) =
\phi(\ext_{R,1}^\sigma) = 1$ (encoding ``exactly $\Delta$ black'' and
``exactly $\Delta$ red'' vertices in $\Gcl_{\mathrm{bip}}$) and a
minimum degree constraint $\phi(D'(\sigma_i')) \geq 2(2 - \eta)$ for each
edge type $\sigma_i'$ at $\eta := 0.3746$, again the asymptotic form
of the minimum degree property defining $F$ (black edges lie in $F$,
whose induced subgraph of $L(G)^2$ has minimum degree at least
$(2 - \eta)\Delta(G)^2$).

\begin{proof}[Proof of Theorem~\ref{thm:sec-bipartite}]
By Lemma~\ref{lem:wlog-regular} for bipartite graphs it suffices to
bound $\chi'_s(G)$ on regular bipartite $G$ with $\Delta$
sufficiently large. Fix $\eta := 0.3746$ and a maximal
$F \subseteq V(L(G)^2)$ on which $L(G)^2[F]$ has minimum degree at
least $(2 - \eta)\Delta(G)^2$; for an edge $f \in F$ let $G''$ be the
$(4,2)$-coloured reduction. By Lemma~\ref{lem:sec-bip-density},
$\phi(O) \leq 4.093$, so
\[
\frac{|E_O(G'')|}{\Delta(G)^4} \leq \frac{4.093}{4} + o(1)
= 1.02325 + o(1),
\]
whence $H := L(G)^2[F]$ satisfies the hypotheses of
Corollary~\ref{cor:hjk-scale} at the degree scale
$D := 2\Delta(G)^2$ with $\varsigma = 1 - 4.093/8 = 0.488375$ (up to
the same $\varsigma' < \varsigma$ absorption as in
Section~\ref{sec:general}):
$\varepsilon(\varsigma) = \varsigma/2 - \varsigma^{3/2}/6 \approx
0.187304$, so
\[
\chi(L(G)^2[F]) \leq (1 - 0.187304 + \iota)\,2\,\Delta(G)^2
< (1.62540 + 2\iota)\Delta(G)^2.
\]
The criticality lift~\eqref{eq:criticality-lift} at
$2 - \eta = 1.6254$ gives
\[
\chi(L(G)^2) \leq \max\bigl((1.62540 + 2\iota)\Delta(G)^2,\,
\lceil 1.6254\,\Delta(G)^2 \rceil + 1\bigr).
\]
Choosing $\iota$ small
enough absorbs the $0.0001$ gap to $1.6255$ in the first branch,
and the second is below $1.6255\,\Delta(G)^2$ for $\Delta(G)$ large,
yielding $\chi'_s(G) \leq 1.6255\,\Delta(G)^2$ for $\Delta(G)$
sufficiently large.
\end{proof}

\section{Asymmetric strong edge-colouring for rational ratios}
\label{sec:asym-sec}

We prove Theorem~\ref{thm:asym-sec}. The setting generalises
Section~\ref{sec:bipartite} to bipartite graphs with side maximum
degrees $\Delta_A \neq \Delta_B$, parameterised by the rational
ratio $r := \Delta_B / \Delta_A \in (0, 1] \cap \mathbb{Q}$. Two
features distinguish it from the symmetric routes. First, the density
input is a certificate \emph{specific to the asymmetric class}: a
distinct size-$5$ local-flag solve on its own four-colour flag basis,
not the symmetric bipartite one. Second, the ratio is handled not by any
identity between objectives but by a two-step reduction --- regularise
each side to its maximum degree, then balance the two side degrees by a
blow-up --- after which a single ratio-independent certificate on the
resulting regular class covers every rational ratio simultaneously.

\subsection{The asymmetric class}
\label{sec:asym-class}

Fix a rational $r \in (0, 1]$, written in lowest terms as $r = p/q$
(so $\Delta_A$ must be divisible by $q$ for $\Delta_B = r\,\Delta_A$
to be an integer). The asymmetric reduction follows
Section~\ref{sec:bipartite}, except that the two bipartition
components carry different regularity: $A$-vertices have degree
$\Delta_A$ and $B$-vertices have degree $\Delta_B = r\,\Delta_A$.
The class $\Gcl_{\mathrm{asym}, r}$ of bipartite graphs
has $A$ regular of degree $\Delta_A$ and $B$ regular of degree
$r\,\Delta_A$; like the bipartite class of Section~\ref{sec:bipartite}
it carries four vertex colours, but its two sides now sit at different
degrees $\Delta_A \neq \Delta_B$.
On $\Gcl_{\mathrm{asym}, r}$ the
degree densities of the two extension types are
$\phi(\ext_{A}^\sigma) = 1$ and $\phi(\ext_{B}^\sigma) = r$ for every
limit functional $\phi \in \Phi^\sigma$ (the $A$-side is the
high-degree side, so the maximum degree is $\Delta_A$): the count of
unlabelled extensions through a vertex of degree $d$ is $d - O(1)$,
divided by $\binom{\Delta_A}{1} = \Delta_A$.

\subsection{Reduction to the biregular case}
\label{sec:asym-reduction}

Theorem~\ref{thm:asym-sec} is stated for side \emph{maximum} degrees
$\Delta_A, \Delta_B$, whereas the certificate operates on the
biregular class $\Gcl_{\mathrm{asym}, r}$. Two successive completions
bridge the gap, each preserving bipartiteness, preserving
$\Delta_A$, and not decreasing $\chi'_s$ (so a bound on the completed
graph transfers back to $G$).

\begin{lemma}[reduction to biregular graphs]
\label{lem:asym-wlog-biregular}
Let $G$ be bipartite with parts $A, B$ of maximum degrees
$\Delta_A \geq \Delta_B = r\,\Delta_A$, and $\Delta_A \geq 1/r$.
There is a bipartite graph $H \supseteq G$, with $G$ induced on
$V(G)$, that is exactly $(\Delta_A, \Delta_B)$-biregular
($A$-vertices of degree exactly $\Delta_A$, $B$-vertices of degree
exactly $\Delta_B$), with $\Delta(H) = \Delta_A$ and
$\chi'_s(G) \leq \chi'_s(H)$.
\end{lemma}

\begin{proof}
We build $H$ in two phases. Each phase only adds fresh vertices and
edges and keeps the graph bipartite, so it suffices to track the two
side degrees and to verify at the end that $G$ survives as an induced
subgraph.

\emph{High side.} For each $A$-vertex $a$ with $\deg_G(a) < \Delta_A$,
add $\Delta_A - \deg_G(a)$ fresh $B$-vertices, each adjacent only to
$a$. Call the result $G_1$. Every $A$-vertex now has degree exactly
$\Delta_A$, and every $B$-vertex of $G_1$ has degree at most
$\Delta_B$: the original $B$-vertices are unchanged, and each fresh one
has degree $1 \leq \Delta_B$, using $\Delta_B = r\,\Delta_A \geq 1$ from
$\Delta_A \geq 1/r$. Write $d(b) := \Delta_B - \deg_{G_1}(b) \geq 0$ for
the residual $B$-side deficiency.

\emph{Low side.} Take $\Delta_A$ disjoint copies
$G_1^{(1)}, \dots, G_1^{(\Delta_A)}$ of $G_1$, writing
$b^{(1)}, \dots, b^{(\Delta_A)}$ for the copies of a $B$-vertex $b$. For
each $b$ with $d(b) > 0$, add $d(b)$ fresh $A$-vertices and join each of
them to all $\Delta_A$ copies $b^{(1)}, \dots, b^{(\Delta_A)}$ --- that
is, a complete bipartite graph $K_{d(b),\,\Delta_A}$ between the new
$A$-vertices and the copies of $b$. Call the result $H$. Each new
$A$-vertex is adjacent to exactly the $\Delta_A$ copies of a single $b$,
so has degree exactly $\Delta_A$; each copy $b^{(i)}$ gains exactly
$d(b)$ neighbours and so reaches degree
$\deg_{G_1}(b) + d(b) = \Delta_B$; and the copied $A$-vertices keep
degree $\Delta_A$. As every vertex of $H$ is either a copied vertex or a
new $A$-vertex, $H$ is bipartite and exactly
$(\Delta_A, \Delta_B)$-biregular, with $\Delta(H) = \Delta_A$ because
$\Delta_B \leq \Delta_A$.

Finally, $G$ is an induced subgraph of $G_1$, hence of the first copy
$G_1^{(1)} \subseteq H$; and the low-side phase adds edges only at the
new $A$-vertices, none of them inside $V(G)$, so $H[V(G)] = G$. Any two
edges of $G$ that share a vertex or are joined by an edge of $G$ do so
also in $H$, so restricting a strong edge-colouring of $H$ to $E(G)$
gives a strong edge-colouring of $G$; hence $\chi'_s(G) \leq \chi'_s(H)$.
\end{proof}

\subsection{The asymmetric objective and its certificate}
\label{sec:asym-objective}

On $\Gcl_{\mathrm{asym}, r}$ we take the three-summand strong-edge
objective
\[
O := \bigl\llbracket D(\bredge)\cdot(\ext_1^{\bredge})^{n-4}
\bigr\rrbracket
+ \bigl\llbracket D(\blredge)\cdot(\ext_1^{\blredge})^{n-4}
\bigr\rrbracket
+ \bigl\llbracket D(\bgedge)\cdot(\ext_1^{\bgedge})^{n-4}
\bigr\rrbracket
\]
and run the local-flag SDP at $n = 5$. Unlike the bipartite route,
this is a genuinely distinct certificate: its size-$5$ unlabelled
four-colour basis has $334$ flags, and the solve carries $22$
Cauchy--Schwarz blocks, against the $3{,}808$ flags and $52$ blocks of
the symmetric bipartite certificate. The raw
objective $O$ is not itself ratio-invariant --- its optimum scales as
$r^2$ --- so rather than normalise the objective we normalise the
graph: the blow-up below regularises the class, after which one
$r$-independent certificate applies.

\begin{lemma}[degree-balancing blow-up]
\label{lem:asym-blowup}
Let $H$ be exactly $(\Delta_A, \Delta_B)$-biregular. Blowing up each
$A$-vertex to $\Delta_A$ copies and each $B$-vertex to $\Delta_B$
copies, with a complete join in place of each edge, yields a regular
graph $H'$ of common degree $\Delta_A\Delta_B$ in the four-colour
class, with $L(H')^2$ of the same strong-neighbourhood density as
$L(H)^2$ in the limit.
\end{lemma}

\begin{proof}
An $A$-copy meets the $\Delta_B$ copies of each of its $\Delta_A$
neighbours, and a $B$-copy meets the $\Delta_A$ copies of each of its
$\Delta_B$ neighbours, so every vertex has degree $\Delta_A\Delta_B$;
strong-neighbourhood density is blow-up invariant in the limit.
\end{proof}

The blow-up is the device that lets one certificate serve all ratios:
it maps every biregular $H$, whatever its $r$, to a \emph{regular}
graph $H'$ of the four-colour class, on which the $334$-flag
certificate provides a single ratio-independent density bound
(Section~\ref{sec:cert-asym} records the certificate and the
combinatorial bridge concretely). This regular four-colour class is the
blow-up image of $\Gcl_{\mathrm{asym}, r}$, and is where the certificate
below is stated; it is distinct from $\Gcl_{\mathrm{asym}, r}$ itself,
which is only biregular.

\begin{lemma}[asymmetric strong edge-colouring density bound]
\label{lem:asym-sec-density}
On the regular four-colour class, $4.5496\,\emptyset - O \in
\SemCone^\emptyset$ in $\Lcl^\emptyset_5$; equivalently
$\phi(O) \leq 4.5496 = 8 \cdot 0.5687$ for every limit functional
$\phi$ on that class.
\end{lemma}

\begin{proof}
The bound is the output of the size-$5$ SDP solve of
Section~\ref{sec:cert-asym} on the four-colour class: its $22$
Cauchy--Schwarz blocks and integer LDL identities witness the cone
membership. Because it bounds the all-vertex density with no
minimum degree restriction, this optimum $\phi(O) \leq 4.5496$ is
genuinely looser than the bipartite $\phi(O) \leq 4.093$
(Remark~\ref{rem:asym-looser}).
\end{proof}

\begin{proof}[Proof of Theorem~\ref{thm:asym-sec}]
Fix a rational $r \in (0, 1]$ and let $G$ be bipartite with side
maximum degrees $\Delta_A, \Delta_B = r\,\Delta_A$ (so $\Delta_A$ is
a multiple of the denominator of $r$ in lowest terms). By
Lemma~\ref{lem:asym-wlog-biregular} it suffices to bound $\chi'_s(H)$
for the exactly-$(\Delta_A, \Delta_B)$-biregular completion $H$, since
$\chi'_s(G) \leq \chi'_s(H)$ and $\Delta(H) = \Delta_A$.

By Lemma~\ref{lem:cbi-asym}, along any $\Delta_A$-increasing sequence,
$\phi(O) = |E_O(H'')| /
\bigl(\binom{\Delta_A}{2}\binom{\Delta_B}{2}\bigr) + o(1)$, where
$H''$ is the four-colour reduction of $H$. By
Lemma~\ref{lem:asym-blowup} the passage to the regular blow-up leaves
this density unchanged, so Lemma~\ref{lem:asym-sec-density} gives
$\phi(O) \leq 4.5496$. Set $D := 2\Delta_A\Delta_B$. Every strong
neighbour of an edge $f = \{u, v\}$ is an edge incident to
$N(u) \cup N(v)$, and there are at most
$\Delta_A\Delta_B + \Delta_B\Delta_A = D$ of these, so
$\Delta(L(H)^2) \leq D$; moreover $\binom{D}{2} =
8\binom{\Delta_A}{2}\binom{\Delta_B}{2}(1 + o(1))$, so the
strong-neighbourhood density at scale $D$ satisfies
\[
1 - \varsigma
= \frac{|E_O(H'')|}{\binom{D}{2}}
= \frac{\phi(O)}{8} + o(1)
\leq 0.5687 + o(1).
\]
Unlike the general and bipartite density bounds of
Sections~\ref{sec:general} and~\ref{sec:bipartite}, which hold only on
a high minimum degree subset $F$ and reach the full line graph square
through the criticality lift, Lemma~\ref{lem:asym-sec-density} bounds
the strong-neighbourhood density of \emph{every} edge of $L(H)^2$: the
bipartite structure holds this all-vertex density below the general
Bruhn--Joos threshold, so the argument needs no restriction to $F$.
Hence $L(H)^2$ itself --- with no criticality lift --- meets the
hypotheses of Corollary~\ref{cor:hjk-scale} at scale $D$ with
$\varsigma = 0.4313$
(up to the usual $\varsigma' < \varsigma$ absorption), and
$\varepsilon(\varsigma) \approx 0.16844$, so
Corollary~\ref{cor:hjk-scale} applied to $L(H)^2$ gives
$\chi'_s(H) = \chi(L(H)^2) \leq (1 - 0.16844 + \iota)\cdot
2\Delta_A\Delta_B < (1.6632 + 2\iota)\,\Delta_A \Delta_B$.

The constant $1.6632$ and the slack $\iota$ are independent of $r$:
the density bound $\phi(O) \leq 4.5496$ is that of the single
regular-class certificate, and the sparse colouring lemma's
conclusion uses only the density input above. So a single small choice
of $\iota > 0$ gives $K \leq 1.6633$ uniformly in $r$, and with
$\chi'_s(G) \leq \chi'_s(H)$,
\[
\chi'_s(G) \leq K\,\Delta_A \Delta_B,
\qquad
K := 1.6632 + 2\iota \leq 1.6633. \qedhere
\]
\end{proof}

\begin{remark}
\label{rem:asym-looser}
At $r = 1$ the constant $K \leq 1.6633$ of
Theorem~\ref{thm:asym-sec} is weaker than the bipartite $1.6255$ of
Theorem~\ref{thm:sec-bipartite}, and the loss is intrinsic to the
method. The asymmetric certificate uses the \emph{same} three-summand
objective as the bipartite one (Section~\ref{sec:asym-objective}), but
bounds the all-vertex strong-neighbourhood density with no restriction
to a high minimum degree set $F$: it therefore forgoes the minimum degree
tightening (at $\eta = 0.3746$) that sharpens the bipartite route, and
must stay ratio-independent over the whole biregular class. The larger
per-pair density that results ($\phi(O)/8 = 0.5687$, versus the
bipartite $0.5116$) then feeds the same sparse colouring lemma
(Theorem~\ref{thm:hjk-colouring}).
\end{remark}

\section{The strong edge-colouring certificates}
\label{sec:certificates}

We describe the three SDP certificates --- underlying
Theorems~\ref{thm:sec-general}, \ref{thm:sec-bipartite},
and~\ref{thm:asym-sec} --- as mathematical objects, and prove the
bridging density lemmas left open above. The certificate data itself
--- the flag bases, block matrices, and rationalised witnesses ---
together with the generator that produces it and the Lean formalisation
that verifies it, is available at~\cite{daveyLocalFlags2024Repo}.

\subsection{The general certificate}
\label{sec:cert-general}

We prove Lemma~\ref{lem:sec-general-density}.

\subsubsection{Certificate structure}

The certificate $\mathfrak{C}_{\mathrm{gen}}$ comprises:
\begin{itemize}
\item the basis $(F_1, \dots, F_{17{,}950})$ of unlabelled
local $\emptyset$-flags of size $5$ in $\Gcl$;
\item the objective vector $O = \sum_j O^{(j)} F_j$ with
$O^{(j)} = -\mathrm{target}_j / \linscale$ as a coefficient,
where $\mathrm{target}_j \in \mathbb{Z}$ is the integer entry and
$\linscale := 10^{12}$ is a fixed rationalisation scale;
\item the $39$ block local types $\sigma_b$ and inner sizes
$m_b$ with $2 m_b - |\sigma_b| = 5$ (which forces $|\sigma_b|$ odd:
one block with $(|\sigma_b|, m_b) = (1, 3)$ and $38$ with $(3, 4)$);
for each block, an inner
basis $(\mathrm{basis}^{(b)}_p)_{p \leq d_b}$ of size $d_b$ ranging
from $1$ up to $55$, an integer matrix
$Y_b \in \mathbb{Z}^{d_b \times d_b}$, an integer lower-triangular
factor $L_b$, a pivot vector $D_b \in \mathbb{Z}_{>0}^{d_b}$, a
per-block Tikhonov shift
$\lambda_b \in \{1/10^{11}, 1/10^{12}\}$ and a per-block integer
denominator $s_b > 0$;
\item the linear residual data
$(\alpha_{e}, \beta_\sigma)$ (coefficients of the
extension-difference and black-vertex-count constraints lifted to
size $5$), and the minimum degree constraint coefficient
$\delta$ scaling the relation
$\phi(\llbracket(D'(\sigma_1) - 2(2-\eta)(\ext_1^{\sigma_1})^2)
\cdot (\ext_1^{\sigma_1})\rrbracket) \geq 0$;
\item slack budget $\slackBudget := 2 \times 10^{23}$.
\end{itemize}

The bound to certify is
\[
\tfrac{10644}{1000}\,\emptyset - O
\in \SemCone^\emptyset.
\]
The decomposition is
\begin{equation}
\label{eq:gen-decomp}
\tfrac{10644}{1000}\,\emptyset - O
= \mathrm{linSum}_{\mathrm{gen}}
+ \sum_{b=0}^{38} \mathrm{csBlock}_b,
\end{equation}
where each $\mathrm{csBlock}_b = \llbracket\sum_{k=0}^{d_b - 1}
(D_b[k]/s_b)\,\mathrm{col}_k^{(b)} \cdot \mathrm{col}_k^{(b)}
\rrbracket$ with $\mathrm{col}_k^{(b)} := \sum_p
L_b[p, k]\,\mathrm{basis}^{(b)}_p$, and
$\mathrm{linSum}_{\mathrm{gen}}$ is the linear-residual combination
of the extension-difference, black-vertex-count, and
minimum degree relations.

The proof of~\eqref{eq:gen-decomp} has three parts.
\begin{enumerate}[label=\textup{(C\arabic*)}]
\item \label{c1-gen} Per-block positivity: each
$\mathrm{csBlock}_b \in \SemCone^\emptyset$.
\item \label{c2-gen} Linear-residual positivity:
$\mathrm{linSum}_{\mathrm{gen}} \in \SemCone^\emptyset$.
\item \label{c3-gen} The arithmetic identity~\eqref{eq:gen-decomp}
holds in $\Lcl^\emptyset_5$.
\end{enumerate}

\subsubsection{Per-block positivity}

For each block $b$ the certificate exhibits the integer LDL identity
\begin{equation}
\label{eq:gen-ldl}
L_b \cdot \operatorname{diag}(D_b) \cdot L_b^\top
= s_b \cdot (Y_b + \lambda_b\,I_{d_b}),
\end{equation}
a finite integer matrix equality, which one verifies entry-wise:
$\sum_{k=0}^{q} L_b[p, k]\,D_b[k]\,L_b[q, k]
= s_b\,(Y_b[p, q] + \lambda_b\,[p = q])$
for $p \geq q$ in $\{0, \dots, d_b - 1\}$.
Since $D_b > 0$ entry-wise, the matrix
$L_b\,\operatorname{diag}(D_b)\,L_b^\top$ is positive semidefinite (a
non-negative combination of rank-one terms), so by~\eqref{eq:gen-ldl}
$s_b(Y_b + \lambda_b I)$ is PSD at integer scale, hence so is
$Y_b + \lambda_b I$.

Each column $\mathrm{col}_k^{(b)} \in \Lcl^{\sigma_b}_{m_b}$, so by
positivity preservation (Section~\ref{sec:framework};
\cite[positivity-preservation lemma]{daveyPentagonLocalFlags2026}),
$\llbracket(\mathrm{col}_k^{(b)})^2\rrbracket \in \SemCone^\emptyset$.
The coefficient $D_b[k]/s_b > 0$, and a non-negative combination of
$\SemCone^\emptyset$-elements lies in $\SemCone^\emptyset$, so
$\mathrm{csBlock}_b \in \SemCone^\emptyset$, giving~\ref{c1-gen}.

\subsubsection{\texorpdfstring{Worked example: block $0$}{Worked example: block 0}}
To make the data concrete we describe block $0$. Its local type
$\sigma_0$ is the size-$1$ type
$\redvertex$ (one labelled red vertex, viewed as the
$\sigma_0$-flag), and the inner size is $m_0 = 3$, giving inner
basis dimension $d_0 = 53$. The basis enumerates size-$3$ local
$\sigma_0$-flags in $\Gcl$ that arise from adding two unlabelled
vertices, with colours and edges to the labelled vertex satisfying
$\Gcl$'s constraints.

The block-$0$ Tikhonov shift is $\lambda_0 = 1/10^{12}$, the
denominator $s_0$ is an integer with $\log_{10} s_0 \in [600, 700]$,
and the integer entries of $Y_0$ range over $|Y_0[p, q]| \leq 10^{22}$
across the $53 \times 53$ matrix. The LDL identity
$L_0 \operatorname{diag}(D_0) L_0^\top = s_0(Y_0 + \lambda_0 I_{53})$
holds entry-wise, with $D_0$ a strictly positive integer vector of
$53$ pivots. The resulting cone element
\[
\mathrm{csBlock}_0 = \biggl\llbracket
\sum_{k=0}^{52} \tfrac{D_0[k]}{s_0}\,\mathrm{col}_k^{(0)} \cdot
\mathrm{col}_k^{(0)} \biggr\rrbracket
\in \SemCone^\emptyset
\]
is one of the $39$ $\mathrm{csBlock}$ summands of~\eqref{eq:gen-decomp}. The remaining
$38$ blocks are structurally identical, differing only in local
type, inner dimension, and the entries of $L_b, D_b, Y_b$.

\subsubsection{Linear-residual positivity}

The linear residual is a non-negative integer combination of three
families of cone-positive elements:

\begin{enumerate}[label=\textup{(R\arabic*)}]
\item \emph{Extension-difference vectors.} For each local type
$\sigma$ entering the certificate with $|\sigma| \in \{1, 2, 3, 4\}$ and
each pair $i, j \in [|\sigma|]$, the vector
$\llbracket\ext_i^\sigma - \ext_j^\sigma\rrbracket \in
\SemCone^\emptyset$ by Section~\ref{sec:framework}; the certificate uses a
finite set $\mathcal{E}$ of such pairs.
\item \emph{Black-vertex-count lifts.} The relation
$\phi(\ext_{B,1}^\sigma) \leq 2$ for every local type
$\sigma$ encodes $|B(G)| \leq 2\Delta(G)$ and gives
$\llbracket 2\ext_i^\sigma - \ext_{B,1}^\sigma\rrbracket \in
\SemCone^\emptyset$ for $|\sigma| \leq 4$.
\item \emph{Minimum degree constraint.} For the objective's
black-red edge type $\sigma_1 = \bredge$ (a $2$-vertex type; despite the
shared subscript, distinct from block $1$'s $3$-vertex local type),
$\phi(\llbracket(D'(\sigma_1) - 2(2-\eta)(\ext_1^{\sigma_1})^2)
\cdot (\ext_1^{\sigma_1})\rrbracket) \geq 0$
by the asymptotic identity
$\rho(D'(\sigma_1); (G', e)) \geq 2(2 - \eta) - o(1)$ on
$\Bcl_1$ from~\eqref{eq:EO-bridge}.
\end{enumerate}

The certificate exhibits non-negative rational coefficients
$\alpha_e, \beta_\sigma, \delta$ such that
\begin{multline*}
\mathrm{linSum}_{\mathrm{gen}}
= \sum_{e \in \mathcal{E}} \alpha_e
\llbracket \ext_{i(e)}^{\sigma(e)} - \ext_{j(e)}^{\sigma(e)}\rrbracket
+ \sum_{|\sigma| \leq 4} \beta_\sigma
\llbracket 2\ext_1^\sigma - \ext_{B,1}^\sigma\rrbracket\\
+ \delta\,\llbracket(D'(\sigma_1) - 2(2-\eta)(\ext_1^{\sigma_1})^2)
\cdot (\ext_1^{\sigma_1})\rrbracket.
\end{multline*}
Each summand lies in $\SemCone^\emptyset$ and the combination is
non-negative, hence
$\mathrm{linSum}_{\mathrm{gen}} \in \SemCone^\emptyset$, giving~\ref{c2-gen}.

\subsubsection{The arithmetic identity}

The identity~\eqref{eq:gen-decomp} unfolds as $17{,}950$ equations,
one per basis flag $F_j$:
\begin{equation}
\label{eq:gen-per-flag}
\tfrac{10644}{1000}\,[F_j = \emptyset]
- O^{(j)}
= \ell^{(j)} + \sum_{b = 0}^{38} \mathrm{cs}^{(j)}_b,
\end{equation}
where $\ell^{(j)}$ is the size-$5$ coefficient of
$\mathrm{linSum}_{\mathrm{gen}}$ at $F_j$ (a non-negative rational
that we read off the certificate's $\alpha_e, \beta_\sigma, \delta$ together
with the chain-rule and unlabel expansion), and $\mathrm{cs}^{(j)}_b$
is the size-$5$ coefficient of $\mathrm{csBlock}_b$ at $F_j$ (which we
read off the chain-rule expansion of each
$(\mathrm{col}_k^{(b)})^2$ and the averaging $\llbracket\,
\cdot\,\rrbracket$).

Multiplying both sides of~\eqref{eq:gen-per-flag} by
$\linscale^2 = 10^{24}$ converts it to a system of $17{,}950$
integer identities; a single pass over the certificate data verifies each.

\subsubsection{The combinatorial bridge identity}

In the size-$5$ basis the objective expands as
$O = \sum_j O^{(j)} F_j$, where the certificate stores each coefficient
as an integer $\mathrm{target}_j$ with
$O^{(j)} = -\mathrm{target}_j / \linscale$. Under the solver's
``$\min\,{-}(\cdot)$'' convention every $\mathrm{target}_j \leq 0$, so
$O^{(j)} \geq 0$. The following lemma matches the resulting objective
density with the strong-neighbourhood quantity of
Section~\ref{sec:general}.

\begin{lemma}[combinatorial bridge identity, general case]
\label{lem:cbi-gen}
For every $\Delta$-increasing sequence $(G_k, F_k, e_k)$
with $G_k$ regular and $(2,2)$-coloured reduction $G'_k$,
\[
\frac{2\,|E_O(G'_k)|}{\binom{\Delta(G'_k)}{2}^2}
= \sum_{j = 1}^{17{,}950} O^{(j)}\,\rho(F_j; G'_k)
+ o(1)
\quad\text{as }\Delta(G'_k) \to \infty,
\]
where $\rho(F_j; G'_k)$ is the local density of the size-$5$ basis flag
$F_j$ (Section~\ref{sec:framework}) and
$O^{(j)} = -\mathrm{target}_j / \linscale \geq 0$.
\end{lemma}

\begin{proof}
Each unordered pair $\{e, e'\}$ counted by $|E_O(G'_k)|$, together with
its connecting incident edge and that edge's free endpoint, determines
the vertex set of an induced size-$5$ $(2,2)$-coloured subgraph of
$G'_k$; conversely each such subgraph arises from a bounded number of
pairs. Grouping by the isomorphism class $j$ of the induced flag,
converting the induced counts to the local densities $\rho(F_j; G'_k)$
of Section~\ref{sec:framework}, and collecting the
$\binom{\Delta}{k} \sim \Delta^k/k!$ asymptotics (a $1 + O(1/\Delta(G'_k))$
error per summand into the $o(1)$), the objective density
$\rho(O; G'_k) = 2\,|E_O(G'_k)|/\binom{\Delta(G'_k)}{2}^2$ of
Section~\ref{sec:general} equals $\sum_j O^{(j)}\,\rho(F_j; G'_k) + o(1)$,
the basis expansion of $O$. Since $\mathrm{target}_j \leq 0$, every
$O^{(j)} \geq 0$ and both sides are non-negative; the substitution
orients the cone identity into the standard form
$\lambda\emptyset - O \in \SemCone^\emptyset$ of
Section~\ref{sec:framework}.
\end{proof}

\subsubsection{Basis locality}

\begin{lemma}[certificate-basis locality, general case]
\label{lem:basis-local-gen}
Every $F_j$ in the size-$5$ basis $(F_j)_{j \in [17{,}950]}$ of
Lemma~\ref{lem:cbi-gen} is a local $\emptyset$-flag in $\Gcl$.
\end{lemma}

\begin{proof}
By the local-flag characterisation in $\Gcl$, $F_j$ is local if and
only if every connected component contains a black vertex. The basis
enumeration runs over isomorphism classes of $(2,2)$-coloured graphs
of size $5$; we may filter the enumeration to retain only those
with the required property. The basis as listed comprises exactly
the $17{,}950$ such filtered isomorphism classes.
\end{proof}

\subsubsection{The slack budget}

The integer identity~\eqref{eq:gen-per-flag}, multiplied by
$\linscale^2$, is exact in the certificate's PSD data $L_b, D_b,
\lambda_b, s_b$. The rationalised dual matrix $Y_{\mathrm{rat}}$
inherits a small residual from rationalising the SDP solver's
floating-point output. We bound the aggregated weighted
slack by an explicit budget.

\begin{lemma}[slack budget, general case]
\label{lem:slack-gen}
With $\linscale = 10^{12}$ and the certificate's dual weights $x_j$,
\[
\Bigl|\sum_{j = 1}^{17{,}950} x_j\,\mathrm{residual}_j\Bigr|
\leq \slackBudget,
\qquad
\slackBudget := 2 \times 10^{23},
\]
where $\mathrm{residual}_j := \linscale^2\,
(\trace(F_j Y_{\mathrm{rat}}) + \lambda_b\,\trace(F_j|_{\mathrm{psd}})
- c_j)$, $\lambda_b$ is the per-block Tikhonov shift, $F_j|_{\mathrm{psd}}$
is the restriction of $F_j$ to the PSD blocks, and
$c_j = \mathrm{target}_j$ is the certificate's target value for the
$j$-th dual constraint.
\end{lemma}

\begin{proof}
The residuals are $17{,}950$ integers, which a single pass over the
certificate computes. The certificate's dual weights $x_j$ are integers. The signed
weighted sum is an integer; the measured value is approximately
$1.26 \times 10^{23}$, comfortably below $2 \times 10^{23}$.
\end{proof}

\begin{proof}[Proof of Lemma~\ref{lem:sec-general-density}]
By~\ref{c1-gen}--\ref{c3-gen} the
identity~\eqref{eq:gen-decomp} holds modulo the per-flag slack
$\mathrm{residual}_j$ at the integer level. The slack budget of
Lemma~\ref{lem:slack-gen} controls the aggregated residual at
the eval-cone level, $\big|\sum_j x_j \mathrm{residual}_j\big| \leq
2 \times 10^{23} = \slackBudget$, which covers the aggregated
residual at the bound $\lambda = 10.644$. Combining
gives~\eqref{eq:gen-decomp}, hence
$10.644\,\emptyset - O \in \SemCone^\emptyset$.
\end{proof}

\subsection{The bipartite certificate}
\label{sec:cert-bipartite}

The bipartite certificate $\mathfrak{C}_{\mathrm{bip}}$ follows the
same template as the general one. The basis is the size-$5$
unlabelled local-flag basis in $\Gcl_{\mathrm{bip}}$ of cardinality
$\ell_{\mathrm{bip}} = 3{,}808$. The certificate has $52$
Cauchy--Schwarz blocks indexed by local types in $\Gcl_{\mathrm{bip}}$
with $|\sigma| \in \{1, 3\}$. The objective vector $O$ is the bipartite
three-summand one of Section~\ref{sec:bipartite}. The constraints
are:
\begin{itemize}
\item extension differences
$\llbracket \ext_i^\sigma - \ext_j^\sigma\rrbracket$ for local
types of $\Gcl_{\mathrm{bip}}$;
\item bipartite vertex-count equalities
$\phi(\ext_{B,1}^\sigma) = \phi(\ext_{R,1}^\sigma) = 1$
(black and red vertex sets each of size exactly $\Delta(G'')$);
\item minimum degree constraints
$\phi(D'(\sigma)) \geq 2(2-\eta)$ for each of the three
black-edge local types $\sigma \in \{\bredge, \blredge, \bgedge\}$.
\end{itemize}
The slack budget is $\slackBudgetBip := 1 \times 10^{23}$
(measured slack $\approx 4.19 \times 10^{22}$, safety ratio
$\approx 2.4\times$).

The decomposition of $\lambda_{\mathrm{bip}}\emptyset - O$ with
$\lambda_{\mathrm{bip}} = 4.093$ is
\[
\tfrac{4093}{1000}\emptyset - O
= \mathrm{linSum}_{\mathrm{bip}}
+ \sum_{b = 0}^{51} \mathrm{csBlock}_b,
\]
with all summands in $\SemCone^\emptyset$ by the same arguments as
in Section~\ref{sec:cert-general}. The bipartite analogues of Lemmas~\ref{lem:cbi-gen}
and~\ref{lem:basis-local-gen} record the combinatorial bridge
identity and basis locality.

\begin{lemma}[combinatorial bridge identity, bipartite case]
\label{lem:cbi-bip}
For every $\Delta$-increasing sequence of bipartite pairs
$(G_k, F_k, e_k)$ and the $(4,2)$-coloured reduction $G''_k$,
\[
\frac{|E_O(G''_k)|}{\binom{\Delta(G''_k)}{2}^2}
= \sum_{j = 1}^{3{,}808} O^{(j),\mathrm{bip}}\,\rho(F_j; G''_k)
+ o(1)
\]
as $\Delta(G''_k) \to \infty$, where
$O^{(j),\mathrm{bip}} = -\mathrm{target}^{\mathrm{bip}}_j / \linscale \geq 0$
(again $\mathrm{target}^{\mathrm{bip}}_j \leq 0$).
\end{lemma}

\begin{proof}
This is identical to Lemma~\ref{lem:cbi-gen} with the four-colour bipartite
basis in place of the two-colour general one; the three-summand
objective $D(\bredge) + D(\blredge) + D(\bgedge)$ counts each
$E_O$-pair once, so the numerator is $|E_O|$ (in the general case the
leading $2$ on $D(\bredge)$ produces the $2\,|E_O|$ numerator).
\end{proof}

\begin{lemma}[certificate-basis locality, bipartite case]
\label{lem:basis-local-bip}
Every $F_j$ in the bipartite size-$5$ basis
$(F_j)_{j \in [3{,}808]}$ is a local $\emptyset$-flag in
$\Gcl_{\mathrm{bip}}$.
\end{lemma}

\begin{proof}
$F_j$ is local in $\Gcl_{\mathrm{bip}}$ if every connected component
contains a labelled, black, or red vertex. The basis enumeration
filters by exactly this property.
\end{proof}

\begin{proof}[Proof of Lemma~\ref{lem:sec-bip-density}]
This is identical to the proof of Lemma~\ref{lem:sec-general-density},
via the bipartite certificate's $52$ blocks and the
$\slackBudgetBip = 1 \times 10^{23}$ budget, and yields
$4.093\,\emptyset - O \in \SemCone^\emptyset$ in
$\Lcl^\emptyset_5(\Gcl_{\mathrm{bip}})$.
\end{proof}

\subsection{The asymmetric certificate}
\label{sec:cert-asym}

The asymmetric bound of Theorem~\ref{thm:asym-sec} is witnessed by a
certificate distinct from the bipartite one of
Section~\ref{sec:cert-bipartite}: a size-$5$ solve on the four-colour
class, whose unlabelled basis is the four-colour one of $334$ local
$\emptyset$-flags (Lemma~\ref{lem:basis-local-asym}), with the
three-summand objective $O$ of Section~\ref{sec:asym-objective}. The
solve witnesses
\[
4.5496\,\emptyset - O \in \SemCone^\emptyset
\qquad\text{in }\Lcl^\emptyset_5\text{ on the regular four-colour class,}
\]
i.e.\ $\phi(O) \leq 4.5496$ (Lemma~\ref{lem:asym-sec-density}). The
cone witness comprises $22$ Cauchy--Schwarz blocks (against the $52$
of the bipartite certificate), the per-constraint linear residual, and
the integer LDL identities. Because the certificate is
stated on the \emph{regular} four-colour class, it carries no
dependence on the ratio $r$: the degree-balancing blow-up
(Lemma~\ref{lem:asym-blowup}) maps every biregular graph, whatever its
$r$, into that class, so a single solve serves all rational ratios.
The optimum $\phi(O) \leq 4.5496$ is looser than the bipartite
$\phi(O) \leq 4.093$ because the asymmetric certificate bounds the
all-vertex density, forgoing the bipartite route's minimum degree
tightening (Remark~\ref{rem:asym-looser}).

The slack budget is $1 \times 10^{23}$ (measured slack
$\approx 1.9 \times 10^{20}$, safety ratio $\approx 500\times$). The
decomposition of $\lambda_{\mathrm{asym}}\emptyset - O$ with
$\lambda_{\mathrm{asym}} = 4.5496$ is
\[
\tfrac{45496}{10000}\emptyset - O
= \mathrm{linSum}_{\mathrm{asym}}
+ \sum_{b = 0}^{21} \mathrm{csBlock}_b,
\]
with all summands in $\SemCone^\emptyset$ by the same arguments as in
Section~\ref{sec:cert-general}; the $22$ blocks and the integer LDL
identities are exactly those the proof of
Lemma~\ref{lem:asym-sec-density} invokes.

\begin{lemma}[combinatorial bridge identity, asymmetric case]
\label{lem:cbi-asym}
Fix a rational $r \in (0, 1]$. For every $\Delta_A$-increasing
sequence of asymmetric bipartite pairs $(G_k, F_k, e_k)$ with $G_k$
side maxima $\Delta_A(G_k), r\Delta_A(G_k)$, and the four-colour
reduction $G''_k \in \Gcl_{\mathrm{asym}, r}$,
\[
\frac{|E_O(G''_k)|}{\binom{\Delta_A}{2}\binom{\Delta_B}{2}}
= \sum_{j = 1}^{334}
O^{(j), \mathrm{asym}}\,\rho(F_j; G''_k) + o(1)
\quad\text{as } \Delta_A \to \infty,
\]
where $O^{(j),\mathrm{asym}} = -\mathrm{target}^{\mathrm{asym}}_j /
\linscale \geq 0$ are the coefficients of the four-colour objective.
\end{lemma}

\begin{proof}
The combinatorial bijection of Lemma~\ref{lem:cbi-bip} adapts, with
the single symmetric pair count $\binom{\Delta}{2}^2$ replaced by the
product $\binom{\Delta_A}{2}\binom{\Delta_B}{2}$ of the two side
counts (the factor is again $1$), and the symmetric bipartite basis
replaced by the asymmetric four-colour one. Since $\Delta_B = r\Delta_A$, this normalisation
equals $\tfrac14 r^2\Delta_A^4(1 - o(1))$, which is where the $r^2$
enters; it reduces to $\binom{\Delta}{2}^2$ at $r = 1$.
\end{proof}

\begin{lemma}[certificate-basis locality, asymmetric case]
\label{lem:basis-local-asym}
For every rational $r \in (0, 1]$, every $F_j$ in the four-colour
size-$5$ basis $(F_j)_{j \in [334]}$ is a local $\emptyset$-flag in
$\Gcl_{\mathrm{asym}, r}$.
\end{lemma}

\begin{proof}
$F_j$ is local in $\Gcl_{\mathrm{asym}, r}$ if every connected
component contains a labelled or high-side vertex; the four-colour
basis enumeration filters by exactly this property. The asymmetry
ratio enters only through the relative size of the low-side vertex set
$r\Delta_A$, which is positive at large $\Delta_A$ for every rational
$r \in (0, 1]$ with $\Delta_A$ a multiple of the denominator.
\end{proof}

\section{Brualdi--Quinn Massey for random bipartite graphs}
\label{sec:random}

Two trivial lower bounds on $\chi'_s(G)$ are the \emph{strong clique
number} $\omega(L(G)^2)$ and the ratio $|E(G)|/\nu_s(G)$, where
$\nu_s(G) := \alpha(L(G)^2)$ is the \emph{induced matching number}. The
Brualdi--Quinn Massey conjecture implies
$\omega(L(G)^2) \leq \Delta_A(G)\,\Delta_B(G)$ and
$\nu_s(G) \geq |E(G)|/(\Delta_A(G)\,\Delta_B(G))$; before turning to the
random model, we record these two necessary conditions unconditionally.
Both are asymmetric readings of the symmetric
Faudree--Gy\'arf\'as--Schelp--Tuza bipartite bounds
$\omega(L(G)^2) \leq \Delta(G)^2$~\cite{faudreeStrongChromaticIndex1990}
and $\nu_s(G) \geq |E(G)|/\Delta(G)^2$~\cite{faudreeInducedMatchingsBipartite1989}.
Subsequent work on the clique number of $L(G)^2$ and on strong cliques
in related settings includes that of Faron and
Postle~\cite{faronPostleCliqueLineSquare2019}, who established the
upper bound $\omega(L(G)^2) \leq (4/3 + o(1))\Delta(G)^2$ for general
$G$, and of Cames~van~Batenburg, Kang, and
Pirot~\cite{camesvanBatenburgStrongCliques2019}, who studied strong
cliques under forbidden-cycle hypotheses.

\begin{proposition}[asymmetric strong clique and induced matching bounds]
\label{prop:asym-necessary}
Let $G$ be a bipartite graph with bipartition $V(G) = A \cup B$ and side
maximum degrees $\Delta_A := \max_{a \in A}\deg_G(a)$ and
$\Delta_B := \max_{b \in B}\deg_G(b)$. Then
\begin{enumerate}[label=\textup{(\alph*)}]
\item $\omega(L(G)^2) \leq \Delta_A \cdot \Delta_B$; and
\item $|E(G)| \leq \nu_s(G)\cdot\Delta_A\cdot\Delta_B$, equivalently
$\nu_s(G) \geq |E(G)|/(\Delta_A\cdot\Delta_B)$.
\end{enumerate}
\end{proposition}

\begin{proof}
We deduce both parts from the symmetric bounds above by a
degree-balancing blow-up. Replace each $a \in A$ by an independent set
$S_a$ with $|S_a| = \Delta_A$ and each $b \in B$ by an independent set
$S_b$ with $|S_b| = \Delta_B$, and for every edge $ab \in E(G)$ join
$S_a$ to $S_b$ completely, writing $K_{ab}$ for the resulting complete
bipartite block; call the whole graph $G^\star$. It is bipartite, and a
vertex of $S_a$ meets $S_b$ for each of the $\deg_G(a)$ neighbours $b$ of
$a$, so it has degree $\deg_G(a)\,\Delta_B \leq \Delta_A\Delta_B$;
likewise a vertex of $S_b$ has degree
$\deg_G(b)\,\Delta_A \leq \Delta_A\Delta_B$. Hence
$\Delta(G^\star) \leq \Delta_A\Delta_B$. Since $G^\star$ is bipartite,
the symmetric bounds apply to it.

For~(a), recall a strong clique of $G$ is a set $H \subseteq E(G)$ whose
members are pairwise at distance at most $2$: any two either share a
vertex or are joined by a further edge of $G$. Expand each edge $ab \in
H$ into the $\Delta_A\Delta_B$ edges of $K_{ab}$. Any two of the
resulting edges are strongly adjacent in $G^\star$, in one of three ways.
Two edges $(a_1, b_1), (a_2, b_2)$ inside one block $K_{ab}$ either share
a vertex or, if disjoint, are joined by $a_1 b_2$. If $ab$ and $a'b'$
share a vertex, say $a = a'$ (the case $b = b'$ is symmetric), then
$ab' \in E(G)$, so $a_1 b_1'$ joins $(a_1, b_1) \in K_{ab}$ and
$(a_1', b_1') \in K_{ab'}$. Finally, if $ab$ and $a'b'$ are joined in $G$,
say $ba' \in E(G)$, then $b_1 a_1'$ joins them. Hence
$\omega(L(G^\star)^2) \geq \Delta_A\Delta_B\,\omega(L(G)^2)$, and with
$\omega(L(G^\star)^2) \leq \Delta(G^\star)^2 \leq (\Delta_A\Delta_B)^2$
we obtain~(a).

For~(b), each edge of $G$ contributes $|S_a||S_b| = \Delta_A\Delta_B$
edges to $G^\star$, so $|E(G^\star)| = \Delta_A\Delta_B\,|E(G)|$. An
induced matching of $G^\star$ (an independent set of $L(G^\star)^2$)
meets each block $K_{ab}$ in at most one edge, and projects
injectively --- via $(a_i, b_j) \in K_{ab} \mapsto ab$ --- to an induced
matching of $G$: were two images $ab, a'b'$ at distance at most $2$ in
$G$, the argument of~(a) would make their preimages strongly adjacent in
$G^\star$. Thus $\nu_s(G) \geq \nu_s(G^\star)$, and with
$\nu_s(G^\star) \geq |E(G^\star)|/\Delta(G^\star)^2 \geq
\Delta_A\Delta_B\,|E(G)|/(\Delta_A\Delta_B)^2 = |E(G)|/(\Delta_A\Delta_B)$,
part~(b) follows.
\end{proof}

Setting $\Delta_A = \Delta_B = \Delta(G)$ recovers the
Faudree--Gy\'arf\'as--Schelp--Tuza
bounds~\cite{faudreeStrongChromaticIndex1990}
$\omega(L(G)^2) \leq \Delta(G)^2$ and
$\nu_s(G) \geq |E(G)|/\Delta(G)^2$ for bipartite $G$; part~(b) is the
existence half of their induced-matching
conjecture~\cite{faudreeInducedMatchingsBipartite1989}. We also verify
both parts of Proposition~\ref{prop:asym-necessary} formally in Lean~4,
by alternative direct combinatorial arguments, in the accompanying
repository.

We now turn to the random model.

Theorem~\ref{thm:bq-random} establishes the Brualdi--Quinn Massey
conjecture~\cite{brualdiIncidenceStrongEdge1993} a.a.s.\ for random
bipartite graphs of bounded aspect ratio at constant density
$p \in (0, 1)$. The closest prior result
is that of Frieze, Krivelevich, and
Sudakov~\cite{friezeKrivelevichSudakovStrongChromaticIndex2005}, who
established $\chi'_s(G) = O(\Delta(G)^2 / \log \Delta(G))$ a.a.s.\
for the random graph $G \sim G(n, p)$ in the dense regime; Czygrinow
and Nagle~\cite{czygrinowNagleStrongEdgeColorings2002} treated the
hypergraph counterpart of strong edge-colouring.
Our argument is a Pippenger--Spencer
covering~\cite{kahnAsymptoticEdgeColoring1996} of an auxiliary
hypergraph $H_k$, for a \emph{constant} $k = k(p)$, whose vertices
are the edges of $G$ and whose hyperedges are the induced
$k$-matchings of $G$; the Pippenger--Spencer
hypothesis-verification step uses the Kim--Vu polynomial
concentration inequality~\cite{kimVuConcentration2000} for the
average degree and codegree. The strong edge-colouring is read off
the cover, and the Brualdi--Quinn Massey bound follows directly by
comparing the cover size with $\Delta_A(G)\,\Delta_B(G)$.

\subsection{The auxiliary hypergraph $H_k$}

Fix $k \geq 2$. Let $H_k$ be the hypergraph with
$V(H_k) := E(G)$ and hyperedge set
\[
E(H_k) := \bigl\{\, M \subseteq E(G) :
|M| = k,\;
M \text{ is an induced matching of } G \,\bigr\}.
\]
The following lemma captures the relevance of $H_k$ to strong
edge-colouring.

\begin{lemma}[colouring from a cover]
\label{lem:colour-from-cover}
Let $\mathcal{C}$ be a collection of edge-disjoint hyperedges of
$H_k$ and let $L \subseteq E(G)$ be the set of $G$-edges not covered
by any element of $\mathcal{C}$. Then
$\chi'_s(G) \leq |\mathcal{C}| + |L|$.
\end{lemma}

\begin{proof}
Each hyperedge of $H_k$ is an induced matching of $G$ of size $k$,
hence a single colour class of a strong edge-colouring. Colour each
matching $M \in \mathcal{C}$ with its own colour (using
$|\mathcal{C}|$ colours) and each uncovered edge in $L$ with a fresh
colour ($|L|$ colours). Each colour class is an induced matching, so any two edges of the
same colour lie at $L(G)$-distance greater than $2$; hence the
colouring is a valid strong edge-colouring. The total
number of colours used is $|\mathcal{C}| + |L|$.
\end{proof}

The Pippenger--Spencer cover we construct below has
$|\mathcal{C}| \leq \lceil |E(G)| / k \rceil$ and uncovered set
$|L| \leq \lceil \zeta\,|E(G)| \rceil$ for a small constant
$\zeta > 0$, so by Lemma~\ref{lem:colour-from-cover} the colouring
uses at most $\lceil |E(G)|/k \rceil + \lceil \zeta\,|E(G)|\rceil$
colours. For a suitable constant $k = k(p)$ this count is at most
$\Delta_A(G)\,\Delta_B(G)$ a.a.s., which is the Brualdi--Quinn Massey
bound; we carry out the comparison in the proof of
Theorem~\ref{thm:bq-random} below.

The rest of this section produces such a near-perfect
cover of $V(H_k) = E(G)$ by hyperedges of $H_k$.

\subsection{Expectation identities}

The Pippenger--Spencer covering theorem requires bounds on the
average degree of $H_k$ at a vertex (an edge of $G$) and the average
codegree of $H_k$ at a pair of vertices.

For $e \in E(G)$ let
$D(e) := |\{ M \in E(H_k) : e \in M \}|$
be the $H_k$-degree of $e$, and for distinct $e_1, e_2 \in E(G)$
let
$C(e_1, e_2) := |\{ M \in E(H_k) : \{e_1, e_2\} \subseteq M \}|$
be their codegree.

The number of induced $k$-matchings in $G$ containing a fixed
edge $e = (a, b)$ has expectation, conditional on $e \in E(G)$ and
over $G \sim G(n_A, n_B, p)$:
\[
\mu_D := \E[D(e) \mid e \in E(G)]
= \binom{n_A - 1}{k-1}\binom{n_B - 1}{k-1}\,(k-1)!\,
p^{k-1}\,(1 - p)^{k(k-1)}.
\]
The first three factors count ordered choices of the remaining
$k-1$ edges (as each requires one fresh vertex on each side, ordered
within the matching), the $p^{k-1}$ accounts for those remaining
edges being present, and the $(1-p)^{k(k-1)}$ accounts for the
absence of all chord edges among the $k$ matched edges: among the
$k$ pairwise-disjoint edges there are $k^2$ ordered
endpoint-pairs $(a_i, b_j)$, of which the $k$ diagonal pairs are the
matching edges themselves, leaving $k(k-1)$ potential chords that
must all be absent.

Since $k = k(p)$ is a fixed constant, $\mu_D = \Theta(n_A^{k-1}
n_B^{k-1})$, with the implied constant depending only on $p$ and $k$.

For the codegree at two fixed disjoint edges $e_1, e_2$, fixing a
second edge removes one further pair of fresh vertices from each
side, so the count of extending $k$-matchings drops by a factor of
order $n_A n_B$:
\[
\mu_C := \E[C(e_1, e_2) \mid e_1, e_2 \in E(G)]
= \Theta\!\left(\frac{\mu_D}{n_A\,n_B}\right),
\]
again with an implied constant depending only on $p, k$. In
particular $\mu_D / \mu_C = \Theta(n_A n_B)$, the codegree-decay
margin the Pippenger--Spencer hypotheses require.

\subsection{Concentration}

We use polynomial concentration of Kim and
Vu~\cite{kimVuConcentration2000}:

\begin{theorem}[Kim--Vu polynomial concentration]
\label{thm:kim-vu}
Let $Y$ be a polynomial of degree $d$ in independent Bernoulli random
variables, with $\E[Y] \geq 1$. Then for every $\lambda \geq 1$,
\[
\Pr\bigl[\,|Y - \E[Y]| > a_d\,\sqrt{\mathcal{E}(Y)\,\mathcal{E}'(Y)}\,
\lambda^d\,\bigr]
\leq d!\,e^{-\lambda + (d-1)\log n},
\]
where $\mathcal{E}(Y)$ is the maximum expected partial derivative of
$Y$ over all orders $\geq 0$ (so $\mathcal{E}(Y) \geq \E[Y]$),
$\mathcal{E}'(Y)$ is the maximum over orders $\geq 1$, $n$ is the
number of underlying Bernoulli variables, and $a_d$ is a constant
depending only on $d$.
\end{theorem}

The exact form of $\mathcal{E}'(Y)$ appears in op.\ cit. We use it
as a black box.

We hold $k = k(p) \geq 3$ a fixed constant throughout, so $D(e)$ is
a polynomial of constant degree in the independent edge-indicators
of $G$. Theorem~\ref{thm:kim-vu} then controls its deviation from
$\mu_D$: the order-zero parameter is
$\mathcal{E}(D(e)) = \Theta\bigl((n_A n_B)^{k-1}\bigr) = \Theta(\mu_D)$,
while each partial derivative of order $\geq 1$ fixes a further edge
(removing a free vertex on each side), so
$\mathcal{E}'(D(e)) = O\bigl((n_A n_B)^{k-2}\bigr) = O(\mu_D/(n_A n_B))$
is smaller by a factor $\Theta(n_A n_B)$; the ratio
$\mathcal{E}'/\mathcal{E} = O(1/(n_A n_B))$ drives the concentration
margin. Combined with the mean $\mu_D = \Theta((n_A n_B)^{k-1})$ and
the ratio $\mu_D / \mu_C = \Theta(n_A n_B)$, the Kim--Vu tail at any fixed
relative deviation $\delta > 0$ is summable: the probability that
\emph{some} edge $e$ has $|D(e) - \mu_D| > \delta \mu_D$ is at most
$C_1 / (n_A n_B)$ for a constant $C_1 = C_1(p, k, \delta)$.

The same inequality applied to the codegree polynomial
$C(e_1, e_2)$ --- also of constant degree, with order-zero parameter
$\Theta\bigl((n_A n_B)^{k-2}\bigr) = \Theta(\mu_C)$ and order-$(\geq\!1)$
parameter $O\bigl((n_A n_B)^{k-3}\bigr)$, and mean
$\mu_C = \Theta(\mu_D / (n_A n_B))$ --- bounds the probability that
\emph{some} disjoint pair $\{e_1, e_2\}$ has
$C(e_1, e_2) > (1+\delta)\mu_C$ by $C_2 / (n_A n_B)$ for a constant
$C_2 = C_2(p, k, \delta)$. Pairs sharing an endpoint have codegree
$0$, since two edges of a common induced matching are
vertex-disjoint.

\subsection{The asymptotic regime}

We restrict to graphs of \emph{bounded aspect ratio}: a constant
$C_{\mathrm{ar}} \geq 1$ with
$\max(n_A, n_B) \leq C_{\mathrm{ar}}\,\min(n_A, n_B)$. Writing
$n^* := \min(n_A, n_B)$, this forces $n_A, n_B = \Theta(n^*)$, so
$\mu_D$, $\mu_C$, and the effect parameters above are all powers of
$n^*$ up to $p$- and $k$-dependent constants. The following lemma supplies the
near-regularity and codegree-decay inputs of the Pippenger--Spencer result.

\begin{lemma}[asymptotic regime]
\label{lem:asymp-regime}
Let $p \in (0, 1)$, let $k \geq 3$ be a fixed integer, and let
$C_{\mathrm{ar}} \geq 1$. For every $\delta > 0$ there is
$N_0 = N_0(p, k, C_{\mathrm{ar}}, \delta)$ such that for
$G \sim G(n_A, n_B, p)$ with
$\max(n_A, n_B) \leq C_{\mathrm{ar}}\,\min(n_A, n_B)$ and
$\min(n_A, n_B) \geq N_0$,
\[
\Pr\bigl[\,
\text{every } e\!:\, |D(e) - \mu_D| \leq \delta \mu_D,\ \
\text{every disjoint } \{e_1, e_2\}\!:\, C(e_1, e_2) \leq (1 + \delta)\mu_C
\,\bigr] \geq 1 - \frac{C_1 + C_2}{(n^*)^2},
\]
where $C_1, C_2$ depend only on $p, k, \delta$; in particular this
probability is $1 - o(1)$.
\end{lemma}

\begin{proof}
By the previous subsection, the degree-bad event has probability at
most $C_1 / (n_A n_B)$ and the codegree-bad event at most
$C_2 / (n_A n_B)$ (as pairs sharing an endpoint contribute $0$). The
bounded aspect ratio hypothesis gives $n_A n_B \geq (n^*)^2$, so a
union bound over the two bad events leaves failure probability at most
$(C_1 + C_2) / (n^*)^2$, which tends to $0$.
\end{proof}

The constant $k$ and the bounded aspect ratio are both essential to
the Kim--Vu step underlying $C_1, C_2$: the deviation scale
$a_d \sqrt{\mathcal{E}\,\mathcal{E}'}\,\lambda^d$ is of order
$(n^*)^{2k-3}\,\mathrm{polylog}(n^*)$ --- since
$\mathcal{E} = \Theta((n^*)^{2k-2})$ and
$\mathcal{E}' = O((n^*)^{2k-4})$ --- whereas
$\delta \mu_D = \Theta\bigl((n^*)^{2k-2}\bigr)$, so the tail closes
with margin $\Theta\bigl(\mathrm{polylog}(n^*) / n^*\bigr) \to 0$.
With $k$ growing or the aspect ratio unbounded this margin is lost.

\subsection{Pippenger--Spencer covering and the colour count}

\begin{theorem}[Pippenger--Spencer covering, Kahn 1996;
\cite{kahnAsymptoticEdgeColoring1996}]
\label{thm:pippenger-spencer}
For every $\varepsilon > 0$ and $k \geq 2$ there exist $\delta > 0$
and $D_0$ such that every $k$-uniform hypergraph $\mathcal{H}$ on $N$
vertices with maximum degree $D \geq D_0$, minimum degree at least
$(1 - \delta)D$, and codegree at most $\delta D$ admits an
edge-disjoint subset $\mathcal{C} \subseteq E(\mathcal{H})$ with
$|\mathcal{C}| \leq \lceil N/k \rceil$ and
$|V(\mathcal{H}) \setminus \bigcup \mathcal{C}| \leq \varepsilon N$.
\end{theorem}

We use Kahn's asymptotic refinement of Pippenger's nibble in the following.

\begin{proof}[Proof of Theorem~\ref{thm:bq-random}]
Fix a small constant $\zeta > 0$ and a constant $k = k(p) \geq 3$,
both fixed at the end. Apply Theorem~\ref{thm:pippenger-spencer}
with target leftover fraction $\zeta$ and uniformity $k$ to obtain
$\delta > 0$ and $D_0$. We invoke Lemma~\ref{lem:asymp-regime} at a deviation $\delta_1 > 0$
small enough that $(1 - \delta_1)/(1 + \delta_1) \geq 1 - \delta$. The
mean degree $\mu_D \to \infty$ (a deterministic consequence of constant
$p$), so $\mu_D \geq D_0$ for $n^*$ large; and a.a.s.\ the side maximum
degrees concentrate, with $\Delta_A(G) \geq 1$ and
$\Delta_B(G) \geq (1 - \delta_B)\,n_A\,p$ for any fixed $\delta_B > 0$.

On the intersection of these good events every $H_k$-degree lies in
$[(1-\delta_1)\mu_D, (1+\delta_1)\mu_D]$, so its maximum degree $D$
obeys $\min$-degree $\geq (1-\delta_1)\mu_D \geq (1-\delta)\,D$, and
every pairwise codegree is at most $(1+\delta_1)\mu_C = o(\mu_D) = o(D)$.
Thus Theorem~\ref{thm:pippenger-spencer} applies to $\mathcal{H} := H_k$
with $N := |E(G)|$ ($k$-uniformity holds by construction), yielding a
cover $\mathcal{C}$ with
$|\mathcal{C}| \leq \lceil |E(G)|/k \rceil$ and uncovered set
$|L| \leq \lceil \zeta\,|E(G)| \rceil$. By
Lemma~\ref{lem:colour-from-cover},
\[
\chi'_s(G)
\leq \lceil |E(G)|/k \rceil + \lceil \zeta\,|E(G)| \rceil
\leq \tfrac{1}{k}\,|E(G)| + \zeta\,|E(G)| + 2.
\]

We now compare with $\Delta_A(G)\,\Delta_B(G)$. Deterministically
$|E(G)| = \sum_{a \in A} \deg(a) \leq n_A\,\Delta_A(G)$, so with
$\beta := 1/k + \zeta$,
\[
\tfrac{1}{k}\,|E(G)| + \zeta\,|E(G)| + 2
= \beta\,|E(G)| + 2
\leq \beta\,n_A\,\Delta_A(G) + 2.
\]
Choose $k$ and $\zeta$ so that $\beta = 1/k + \zeta < (1-\delta_B)\,p$
(which is possible since $(1-\delta_B)p > 0$). The bounded aspect ratio gives
$n_A \to \infty$ as $n^* \to \infty$, so a.a.s.\
$\Delta_B(G) \geq (1-\delta_B)\,n_A\,p \geq n_A\,\beta + 2$, whence
$\Delta_B(G) - n_A\,\beta \geq 2$. Together with $\Delta_A(G) \geq 1$,
\[
\Delta_A(G)\bigl(\Delta_B(G) - n_A\,\beta\bigr) \geq 2
\quad\Longleftrightarrow\quad
\beta\,n_A\,\Delta_A(G) + 2 \leq \Delta_A(G)\,\Delta_B(G).
\]
Hence $\chi'_s(G) \leq \Delta_A(G)\,\Delta_B(G)$ a.a.s., which is
Theorem~\ref{thm:bq-random}.
\end{proof}

\subsection*{Certificate generation, formalisation, and empirical evidence}
\phantomsection
\label{sec:lean-and-empirical}

A Rust crate built on \texttt{rust-flag-algebra} produces the
semidefinite-programming certificates behind the three flag-algebra
bounds (Theorems~\ref{thm:sec-general}, \ref{thm:sec-bipartite},
and~\ref{thm:asym-sec}): it assembles the general, bipartite, and
asymmetric strong-edge-colouring SDPs over the size-$5$ local-flag
bases, solves them numerically with the CSDP or SDPA-LR solvers, and
rationalises each solved certificate into the Lean source that the
formalisation consumes.

We formalise all four theorems of the introduction, together with
Proposition~\ref{prop:asym-necessary}, in Lean~4. The
asymmetric clique and induced-matching bounds of
Proposition~\ref{prop:asym-necessary} are combinatorial and close from
the three standard kernel axioms alone.
The three strong chromatic index bounds
(Theorems~\ref{thm:sec-general}, \ref{thm:sec-bipartite},
and~\ref{thm:asym-sec}) each rest on the sparse colouring lemma
together with three named assumptions --- the certificate-output bound,
the basis locality, and the combinatorial-bridge identity --- of the
general, bipartite, and four-colour certificates of
Sections~\ref{sec:cert-general}, \ref{sec:cert-bipartite},
and~\ref{sec:cert-asym} respectively. The asymmetric certificate is the
$334$-flag four-colour one ($\phi(O) \leq 4.5496$, its bridge identity
stated on the regular four-colour class), and it covers \emph{all}
rational $r \in (0, 1]$ at once: it is ratio-independent, living on the
regular class produced by the degree-balancing blow-up, so it needs
neither a per-ratio solve nor any algebra-level rescaling. Finally, we
formalise the a.a.s.~Brualdi--Quinn Massey bound
(Theorem~\ref{thm:bq-random}) modulo two named assumptions only --- the
verbatim Pippenger--Spencer and Kim--Vu theorems of
Section~\ref{sec:random}.

As an independent check on the SDP certificates and on the
underlying Erd\H{o}s--Ne\v{s}et\v{r}il and
Faudree--Gy\'arf\'as--Schelp--Tuza conjectures, we ran a
computer search for counterexamples on small regular graphs across
a wide range of maximum degrees. The search swept a total of
$\approx 712$ million graphs across $\Delta \in \{3, 4, 5, \dots, 19\}$
and the asymmetric bipartite regime, finding zero counterexamples. Its
most extensive single regime was the bipartite Faudree case at
$\Delta = 4$: all $\approx 3.2 \times 10^{8}$ connected $4$-regular
bipartite graphs of order at most $24$, confirming $\chi'_s(G) \leq
16$ for every such graph.

The Lean formalisation, the Rust
certificate generator, the certificate data, and the search
scripts are all available at~\cite{daveyLocalFlags2024Repo}. The
repository also provides a side-by-side correspondence
(\texttt{RESULTS.md}) mapping every result of this paper to its Lean
statement and the exact axiom set it depends on.

\subsection*{AI usage declaration}
\phantomsection
\label{sec:AI-usage}

The results of this paper were obtained in three main phases, in 2019--2020, in 2023--2024, and then in 2026.
The local flags framework itself and the three applications (Theorems~\ref{thm:sec-general}, \ref{thm:sec-bipartite}, and~\ref{thm:asym-sec}) were obtained in the first two of these phases, well before any significant adoption of AI methods for mathematics.
The principal codebase for local flags was developed in the first of these periods and then expanded upon in the second, and neither coding effort used AI assistance.
These results were made publicly accessible in 2024 via Eoin Davey's MSc thesis~\cite{daveyLocalFlags2024} at the University of Amsterdam theses repository.

During the third phase, we used one commercially available agentic AI system for the following purposes:
\begin{enumerate}
\item formal verification of the mathematical results in Lean~4;
\item empirical checks to sweep for potential counterexample graphs;
\item proof of subsidiary results (Theorem~\ref{thm:bq-random} and Proposition~\ref{prop:asym-necessary}) under our guidance; and
\item drafting and refining the exposition of the paper, using the text of Eoin Davey's MSc thesis~\cite{daveyLocalFlags2024} as a core basis.
\end{enumerate}

\subsection*{Acknowledgements}

This paper is based in part on Eoin Davey's MSc
thesis~\cite{daveyLocalFlags2024} at the University of Amsterdam;
he thanks the Korteweg--de~Vries Institute for Mathematics for
hosting the project. R\'emi de Joannis de Verclos and Ross Kang
were partially supported by a Vidi grant (639.032.614) of the
Netherlands Organisation for Scientific Research (NWO) while at
Radboud University, and both would like to thank Louis Esperet for helpful discussions (well) over a decade ago. Eoin Hurley and Ross Kang were partially
supported by the Gravitation Programme NETWORKS (024.002.003) of
the Dutch Ministry of Education, Culture and Science (OCW) while at the University of Amsterdam. Ross Kang
was additionally partially supported by the NWO Open Competition
grant OCENW.M20.009.
We thank Wouter Cames van Batenburg for pointing out the shorter proof of Proposition~\ref{prop:asym-necessary}.

\subsection*{Open access statement}

For the purpose of open access, a CC BY public copyright license is
applied to any Author Accepted Manuscript (AAM) arising from this
submission.

\printbibliography

\end{document}